%% file: article4.tex
\documentclass[11pt,twoside]{article}

\newcommand{\HeadTitle}{Structure and Zero Asymptotics of Differential Operators Associated with \texorpdfstring{\({\Xi}_n\)}{Xi_n} and \texorpdfstring{\({\Lambda}_n\)}{Lambda_n}}

\newcommand{\HeadTitleTwo}{\begin{center}
\Large{\textit{Structure and Zero Asymptotics of Differential Operators Associated with \texorpdfstring{\({\Xi}_n\)}{Xi_n} and \texorpdfstring{\({\Lambda}_n\)}{Lambda_n}}}
\end{center}}

\input{preamble}
\input{font}

\begin{document}

\input{titlepage}

\begin{abstract}
We study the second-order differential operators \(\mathcal D_{\Xi}\) and \(\mathcal D_{\Lambda}\) associated with the rescaled polynomial families \((\widetilde{\Xi}_n)\) and \((\widetilde{\Lambda}_n)\), and more generally the polynomial sequences generated by iterating these operators from an arbitrary linear initial datum \(cx-d\).

We establish structural properties of \(\mathcal D_{\Xi}\) and \(\mathcal D_{\Lambda}\), including factorizations into first-order operators, weighted divergence forms, formal self-adjointness, and hypergeometric descriptions of the corresponding formal eigenvalue equations. We also show that both operators preserve hyperbolicity, preserve zeros in \((0,b)\) for \(b\ge 1\), and preserve proper position.

For the iterated polynomial sequences, we derive explicit closed formulae in terms of the auxiliary families \((\widetilde{\Xi}_n)\) and \((\widetilde{\Lambda}_n)\), prove strict interlacing of consecutive zeros under explicit conditions on \(d/c\), and obtain asymptotic formulae for the normalized logarithmic derivatives. As a consequence, the associated zero counting measures converge weakly to the same limiting probability measure as in the auxiliary case.
\end{abstract}

\vspace{0.2cm}

\paragraph{Notation.}
Throughout this manuscript we use the following conventions. 
We write $\mathbb{N}_0$ for the set of all non-negative integers and $\mathbb{N}$ for the set of positive integers. 
The symbols $A_n(x)$ and $B_n(x)$ denote respectively the Eulerian polynomials of type ~$A$ and ~$B$. 

\vspace{0.5cm}

\section*{Introduction}

In \cite{TallaWaffo2026arxiv2602.16761}, we introduced two families of even polynomials
\[
\Xi_n,\Lambda_n\in\mathbb{Q}[x],
\qquad n\ge 1,
\]
through integral representations of the normalized values
\[
\frac{\beta(2n)}{\pi^{2n-1}}
\qquad\text{and}\qquad
\frac{\zeta(2n+1)}{\pi^{2n}}.
\]
More precisely, for every \(n\in\mathbb{N}_0\), one has
\[
\frac{\beta(2n)}{\pi^{2n-1}}
=
\int_0^1 \frac{x\,\Xi_n(x)}{\sqrt{1-x^2}\operatorname{arctanh}(x)}\,dx,
\qquad
\frac{\zeta(2n+1)}{\pi^{2n}}
=
\int_0^1 \frac{x\,\Lambda_n(x)}{\operatorname{arctanh}(x)}\,dx.
\]
These polynomial families arise naturally from Malmsten-type integral representations and from rational representations of polylogarithms of negative order in terms of Eulerian polynomials.

\vspace{0.2cm}

A first remarkable feature of the sequences \((\Xi_n)_{n\ge 1}\) and \((\Lambda_n)_{n\ge 1}\) is that all their terms are even polynomials. It is therefore natural to pass from the variable \(x\) to \(x^2\), and to introduce the rescaled polynomial sequences
\[
\widetilde{\Xi}_n(x):=\Xi_n(\sqrt{x}),
\qquad
\widetilde{\Lambda}_n(x):=\Lambda_n(\sqrt{x}),
\]
so that
\[
\Xi_n(x)=\widetilde{\Xi}_n(x^2),
\qquad
\Lambda_n(x)=\widetilde{\Lambda}_n(x^2),
\qquad
\widetilde{\Xi}_n,\widetilde{\Lambda}_n\in\mathbb{Q}[x].
\]

\vspace{0.2cm}

This change of variable transforms the original even-polynomial setting into an ordinary polynomial one and makes visible the differential operators underlying the recursive construction.

\vspace{0.2cm}

The starting point is the pair of differential recurrences established in \cite{TallaWaffo2026arxiv2602.16761}:
\[
\left\{
\begin{aligned}
\Xi_{n+1}(x)
&=
-\frac{1}{8n(2n+1)}
\Biggl[
(x^2-1)^2\Xi_n''(x)
+\frac{2(x^2-1)(3x^2-1)}{x}\Xi_n'(x)
+(6x^2-5)\Xi_n(x)
\Biggr],\\[1.2ex]
\Lambda_{n+1}(x)
&=
-\frac{2^{2n+1}-1}{(2^{2n+3}-1)(2n+1)(2n+2)}
\Biggl[
(x^2-1)^2\Lambda_n''(x)
+\frac{2(x^2-1)(4x^2-1)}{x}\Lambda_n'(x)
+4(3x^2-2)\Lambda_n(x)
\Biggr].
\end{aligned}
\right.
\]
Using
\[
\left\{
\begin{aligned}
\Xi_n(x)&=\widetilde{\Xi}_n(x^2),&
\Xi_n'(x)&=2x\,\widetilde{\Xi}_n'(x^2),&
\Xi_n''(x)&=2\widetilde{\Xi}_n'(x^2)+4x^2\widetilde{\Xi}_n''(x^2),\\[0.8ex]
\Lambda_n(x)&=\widetilde{\Lambda}_n(x^2),&
\Lambda_n'(x)&=2x\,\widetilde{\Lambda}_n'(x^2),&
\Lambda_n''(x)&=2\widetilde{\Lambda}_n'(x^2)+4x^2\widetilde{\Lambda}_n''(x^2),
\end{aligned}
\right.
\]
and replacing \(x^2\) by \(x\), we obtain
\[
\left\{
\begin{aligned}
\widetilde{\Xi}_{n+1}(x)
&=
-\frac{1}{8n(2n+1)}
\Bigl[
4x(x-1)^2\widetilde{\Xi}_n''(x)
+2(x-1)(7x-3)\widetilde{\Xi}_n'(x)
+(6x-5)\widetilde{\Xi}_n(x)
\Bigr],\\[1.2ex]
\widetilde{\Lambda}_{n+1}(x)
&=
-\frac{2^{2n+1}-1}{(2^{2n+3}-1)(2n+1)(2n+2)}
\Bigl[
4x(x-1)^2\widetilde{\Lambda}_n''(x)
+6(x-1)(3x-1)\widetilde{\Lambda}_n'(x)
+4(3x-2)\widetilde{\Lambda}_n(x)
\Bigr].
\end{aligned}
\right.
\]

\vspace{0.2cm}

This leads to the differential operators
\[
\left\{
\begin{aligned}
\mathcal D_{\Xi}[f](x)
&:=
4x(x-1)^2 f''(x)
+2(x-1)(7x-3)f'(x)
+(6x-5)f(x),\\[0.8ex]
\mathcal D_{\Lambda}[f](x)
&:=
4x(x-1)^2 f''(x)
+6(x-1)(3x-1)f'(x)
+4(3x-2)f(x),
\end{aligned}
\right.
\]
for which the rescaled recurrences take the compact form
\[
\widetilde{\Xi}_{n+1}=a_n\,\mathcal D_{\Xi}[\widetilde{\Xi}_n],
\qquad
\widetilde{\Lambda}_{n+1}=b_n\,\mathcal D_{\Lambda}[\widetilde{\Lambda}_n],
\]
where
\[
a_n=-\frac{1}{8n(2n+1)},
\qquad
b_n=-\frac{2^{2n+1}-1}{(2^{2n+3}-1)(2n+1)(2n+2)}.
\]

\vspace{0.2cm}

The purpose of the present paper is to study these operators and the polynomial sequences generated by their iteration. In contrast with the constant initial values considered in \cite{TallaWaffo2026arxiv2602.16761}, we now start from the linear polynomial
\[
cx-d,
\qquad c, d\in\mathbb{R}.
\]

\vspace{0.2cm}

The first part is devoted to the intrinsic properties of the operators \(\mathcal D_{\Xi}\) and \(\mathcal D_{\Lambda}\), including their action on polynomial spaces, their effect on degrees and coefficients, and the structural features they preserve. The second part concerns the polynomial families generated by iterating these operators from the initial datum \(x-c\). Our aim is to understand how the algebraic properties of the operators govern the behavior of the resulting sequences, including their degree growth, coefficient structure, and zero distribution.

\vspace{0.2cm}

This operator-theoretic viewpoint provides a natural continuation of \cite{TallaWaffo2026arxiv2602.16761}. It shifts the emphasis from the explicit description of the original even polynomials to the structural study of the differential mechanisms that generate them and, more generally, of the new polynomial families arising from nonconstant initial data.

\part{Intrinsic properties of the operators \(\mathcal D_{\Xi}\) and \(\mathcal D_{\Lambda}\)}

\section{Basic properties of the operators}

We consider the second-order differential operators
\[
\left\{
\begin{aligned}
\mathcal D_{\Xi}[f](x)
&=
4x(x-1)^2 f''(x)
+2(x-1)(7x-3)f'(x)
+(6x-5)f(x),\\[0.8ex]
\mathcal D_{\Lambda}[f](x)
&=
4x(x-1)^2 f''(x)
+6(x-1)(3x-1)f'(x)
+4(3x-2)f(x).
\end{aligned}
\right.
\]
They act naturally on \(\mathbb R[x]\). Since the coefficients are polynomial and the highest-order coefficient has degree \(3\), both operators map \(\mathbb R_m[x]\) into \(\mathbb R_{m+1}[x]\) for every \(m\ge0\). In fact, they increase the degree of every nonzero polynomial exactly by one.

Their action on monomials is given by
\[
\mathcal D_{\Xi}[x^m]
=
(4m^2+10m+6)x^{m+1}
-(8m^2+12m+5)x^m
+(4m^2+2m)x^{m-1},
\]
and
\[
\mathcal D_{\Lambda}[x^m]
=
(4m^2+14m+12)x^{m+1}
-(8m^2+16m+8)x^m
+(4m^2+2m)x^{m-1},
\]
with the convention that the term in \(x^{m-1}\) is absent when \(m=0\). Consequently, if
\[
f(x)=a_d x^d+\cdots \qquad (a_d\neq0),
\]
then
\[
\deg\bigl(\mathcal D_{\Xi}[f]\bigr)=d+1,
\qquad
\deg\bigl(\mathcal D_{\Lambda}[f]\bigr)=d+1,
\]
and the leading coefficients are
\[
\operatorname{LC}\bigl(\mathcal D_{\Xi}[f]\bigr)
=
(4d^2+10d+6)a_d,
\qquad
\operatorname{LC}\bigl(\mathcal D_{\Lambda}[f]\bigr)
=
(4d^2+14d+12)a_d.
\]

At the special points \(x=0\) and \(x=1\), one finds
\[
\bigl(\mathcal D_{\Xi}[p]\bigr)(0)=-6p'(0)-5p(0),
\qquad
\bigl(\mathcal D_{\Xi}[p]\bigr)(1)=p(1),
\]
and
\[
\bigl(\mathcal D_{\Lambda}[p]\bigr)(0)=-6p'(0)-8p(0),
\qquad
\bigl(\mathcal D_{\Lambda}[p]\bigr)(1)=4p(1).
\]
In particular, both operators preserve the subspace
\[
(x-1)\mathbb R[x]=\{p\in\mathbb R[x]:p(1)=0\}.
\]
By contrast, vanishing at \(x=0\) is not preserved in general, since the values at the origin depend on both \(p(0)\) and \(p'(0)\).

\section{Factorization and formal weighted self-adjoint form}

We next record two structural properties of the operators
\[
\mathcal D_{\Xi}[f](x)
=
4x(x-1)^2 f''(x)
+2(x-1)(7x-3)f'(x)
+(6x-5)f(x),
\]
and
\[
\mathcal D_{\Lambda}[f](x)
=
4x(x-1)^2 f''(x)
+6(x-1)(3x-1)f'(x)
+4(3x-2)f(x).
\]
First, both operators admit factorizations into first-order differential operators. Second, they can be written in weighted divergence form, which yields a natural formal self-adjoint structure.

\vspace{0.2cm}

\begin{proposition}\label{prop:factorization}
For \(n\in\mathbb{N}_0\), define the differential operators \(A_n\) and \(B_n\) by
\[
A_n[f]:=2(x-1)D[f]+nf,
\qquad
B_n[f]:=2x(x-1)D[f]+((n+1)x-1)f,
\]
where \(D=\frac{d}{dx}\). Then
\[
\mathcal D_{\Xi}[f]=(A_1\circ B_1)[f],
\qquad
\mathcal D_{\Lambda}[f]=(A_2\circ B_2)[f].
\]
Equivalently,
\[
\mathcal D_{\Xi}=A_1\circ B_1,
\qquad
\mathcal D_{\Lambda}=A_2\circ B_2.
\]
\end{proposition}

\begin{proof}
We compute, for \(n\in\mathbb{N}_0\),
\[
(A_n\circ B_n)[f]
=
A_n[B_n[f]]
=
2(x-1)D[B_n[f]]+n\,B_n[f].
\]
Since
\[
B_n[f]=2x(x-1)D[f]+((n+1)x-1)f,
\]
we obtain
\[
D[B_n[f]]
=
D\!\bigl[2x(x-1)D[f]+((n+1)x-1)f\bigr]
=
2x(x-1)D^2[f]+((n+5)x-3)D[f]+(n+1)f.
\]
Hence
\[
(A_n\circ B_n)[f]
=
2(x-1)\bigl(2x(x-1)D^2[f]+((n+5)x-3)D[f]+(n+1)f\bigr)
+
n\bigl(2x(x-1)D[f]+((n+1)x-1)f\bigr).
\]
Therefore,
\[
(A_n\circ B_n)[f]
=
4x(x-1)^2D^2[f]
+
2(x-1)\bigl((2n+5)x-3\bigr)D[f]
+
\bigl((n+1)(n+2)x-(3n+2)\bigr)f.
\]
For \(n=1\), this gives
\[
(A_1\circ B_1)[f]
=
4x(x-1)^2D^2[f]
+
2(x-1)(7x-3)D[f]
+
(6x-5)f
=
\mathcal D_{\Xi}[f].
\]
For \(n=2\), we obtain
\[
(A_2\circ B_2)[f]
=
4x(x-1)^2D^2[f]
+
6(x-1)(3x-1)D[f]
+
4(3x-2)f
=
\mathcal D_{\Lambda}[f].
\]
This proves the result.
\end{proof}

\begin{remark}
The above factorizations reduce both second-order operators to compositions of first-order operators. This will be useful in the analysis of structural properties of the polynomial sequences generated by iteration.
\end{remark}

\vspace{0.2cm}

We now show that both operators admit Sturm--Liouville-type representations.

\begin{proposition}\label{prop:weighted-form}
The operator \(\mathcal D_{\Xi}\) can be written in the form
\[
\mathcal D_{\Xi}[f]
=
x^{-1/2}
\Bigl(
\bigl(4x^{3/2}(x-1)^2f'\bigr)'
+
x^{1/2}(6x-5)f
\Bigr),
\]
whereas \(\mathcal D_{\Lambda}\) admits the representation
\[
\mathcal D_{\Lambda}[f]
=
\bigl(x^{1/2}(1-x)\bigr)^{-1}
\Bigl(
\bigl(4x^{3/2}(1-x)^3f'\bigr)'
+
4x^{1/2}(1-x)(3x-2)f
\Bigr).
\]
\end{proposition}

\begin{proof}
We begin with \(\mathcal D_{\Xi}\). Write \(\mathcal D_{\Xi}=A(x)D^2+B(x)D+C(x)\), where
\[
A(x)=4x(x-1)^2,\qquad B(x)=2(x-1)(7x-3),\qquad C(x)=6x-5.
\]
We seek a weight \(\rho_{\Xi}\) such that \((\rho_{\Xi}A)'=\rho_{\Xi}B\). Since
\[
A'(x)=4(x-1)(3x-1),
\]
it follows that
\[
\frac{\rho_{\Xi}'}{\rho_{\Xi}}
=
\frac{B-A'}{A}
=
\frac{2(x-1)(7x-3)-4(x-1)(3x-1)}{4x(x-1)^2}
=
\frac{1}{2x}.
\]
Hence \(\rho_{\Xi}(x)=x^{1/2}\), up to a nonzero multiplicative constant. Therefore
\[
\mathcal D_{\Xi}[f]
=
\rho_{\Xi}^{-1}
\Bigl(
(\rho_{\Xi}A f')'
+
\rho_{\Xi}Cf
\Bigr),
\]
that is,
\[
\mathcal D_{\Xi}[f]
=
x^{-1/2}
\Bigl(
\bigl(4x^{3/2}(x-1)^2f'\bigr)'
+
x^{1/2}(6x-5)f
\Bigr).
\]

For \(\mathcal D_{\Lambda}\), write \(\mathcal D_{\Lambda}=A(x)D^2+B(x)D+C(x)\), where
\[
A(x)=4x(x-1)^2,\qquad B(x)=6(x-1)(3x-1),\qquad C(x)=4(3x-2).
\]
Again we require \((\rho_{\Lambda}A)'=\rho_{\Lambda}B\). Using
\[
A'(x)=4(x-1)(3x-1),
\]
we find
\[
\frac{\rho_{\Lambda}'}{\rho_{\Lambda}}
=
\frac{B-A'}{A}
=
\frac{6(x-1)(3x-1)-4(x-1)(3x-1)}{4x(x-1)^2}
=
\frac{3x-1}{2x(x-1)}.
\]
Thus \(\rho_{\Lambda}(x)=x^{1/2}(1-x)\), again up to a nonzero multiplicative constant. Hence
\[
\mathcal D_{\Lambda}[f]
=
\rho_{\Lambda}^{-1}
\Bigl(
(\rho_{\Lambda}A f')'
+
\rho_{\Lambda}Cf
\Bigr),
\]
which yields
\[
\mathcal D_{\Lambda}[f]
=
\bigl(x^{1/2}(1-x)\bigr)^{-1}
\Bigl(
\bigl(4x^{3/2}(1-x)^3f'\bigr)'
+
4x^{1/2}(1-x)(3x-2)f
\Bigr).
\]
This proves the proposition.
\end{proof}

\begin{corollary}\label{cor:formal-selfadjoint}
The operator \(\mathcal D_{\Xi}\) is formally self-adjoint with respect to the weighted inner product
\[
\langle f,g\rangle_{\Xi}
=
\int_0^1 f(x)g(x)\,x^{1/2}\,dx,
\]
whereas \(\mathcal D_{\Lambda}\) is formally self-adjoint with respect to
\[
\langle f,g\rangle_{\Lambda}
=
\int_0^1 f(x)g(x)\,x^{1/2}(1-x)\,dx.
\]
\end{corollary}

\begin{proof}
This follows immediately from the divergence-form representations established in \cref{prop:weighted-form}.
\end{proof}

\section{The eigenvalue equations}

We now collect several analytic features of the operators
\[
\mathcal D_{\Xi}[f](x)
=
4x(x-1)^2 f''(x)
+2(x-1)(7x-3)f'(x)
+(6x-5)f(x),
\]
and
\[
\mathcal D_{\Lambda}[f](x)
=
4x(x-1)^2 f''(x)
+6(x-1)(3x-1)f'(x)
+4(3x-2)f(x).
\]
In particular, we describe their weighted divergence form, the resulting formal self-adjointness, and the associated formal eigenvalue equations.

\subsection{Orthogonality of eigenfunctions}

The main consequence of formal self-adjointness is the orthogonality of eigenfunctions corresponding to distinct eigenvalues, provided the boundary term vanishes.

\begin{proposition}\label{prop:orthogonality}
Let \(\mathcal D\) be one of the operators \(\mathcal D_{\Xi}\) or \(\mathcal D_{\Lambda}\), and let \(\rho\) denote the corresponding weight. Suppose that \(f_n\) and \(f_m\) satisfy
\[
\mathcal D[f_n]=\lambda_n f_n,
\qquad
\mathcal D[f_m]=\lambda_m f_m,
\]
with \(\lambda_n\neq\lambda_m\). Assume moreover that the boundary term
\[
\Bigl[p(x)\bigl(f_mf_n'-f_nf_m'\bigr)\Bigr]_{0}^{1}
\]
vanishes, where \(p(x)\) is the coefficient of \(f'\) in the divergence form
\[
\mathcal D[f]
=
\rho(x)^{-1}\bigl((p(x)f'(x))'+q(x)f(x)\bigr).
\]
Then
\[
\int_0^1 \rho(x)\,f_n(x)f_m(x)\,dx=0.
\]
\end{proposition}

\begin{proof}
Write
\[
(pf_n')'+qf_n=\lambda_n \rho f_n,
\qquad
(pf_m')'+qf_m=\lambda_m \rho f_m.
\]
Multiplying the first equation by \(f_m\), the second by \(f_n\), and subtracting, we obtain
\[
f_m(pf_n')'-f_n(pf_m')'=(\lambda_n-\lambda_m)\rho f_nf_m.
\]
Integrating over \((0,1)\) yields
\[
\Bigl[p(x)\bigl(f_mf_n'-f_nf_m'\bigr)\Bigr]_{0}^{1}
=
(\lambda_n-\lambda_m)\int_0^1 \rho(x)f_n(x)f_m(x)\,dx.
\]
If the boundary term vanishes and \(\lambda_n\neq\lambda_m\), the conclusion follows.
\end{proof}

\subsection{The eigenvalue equation for \texorpdfstring{$\mathcal D_{\Xi}$}{DXi}}

We now consider the formal eigenvalue problem \(\mathcal D_{\Xi}f=\lambda f\), that is,
\[
4x(1-x)^2f''+2(x-1)(7x-3)f'+(6x-5-\lambda)f=0.
\]
To reduce this equation to hypergeometric form, set \(f(x)=(1-x)^\alpha u(x)\). A direct computation shows that the equation reduces to Gauss's hypergeometric equation provided \(\lambda=(2\alpha+1)^2\). Then \(u\) satisfies
\[
x(1-x)u''+\left(\frac32-\frac{4\alpha+7}{2}x\right)u'
-(\alpha+1)\left(\alpha+\frac32\right)u=0,
\]
which is of the standard form
\[
x(1-x)u''+[c-(a+b+1)x]u'-abu=0
\]
with parameters
\[
a=\alpha+1,\qquad b=\alpha+\frac32,\qquad c=\frac32.
\]

\begin{proposition}\label{prop:eigen-DXi}
The general formal solution of the eigenvalue equation \(\mathcal D_{\Xi}f=\lambda f\) is
\[
f(x)=(1-x)^\alpha\left[
C_1\,{}_2F_1\!\left(\alpha+1,\alpha+\frac32;\frac32;x\right)
+
C_2\,x^{-1/2}\,{}_2F_1\!\left(\alpha+\frac12,\alpha+1;\frac12;x\right)
\right],
\]
where \(\lambda=(2\alpha+1)^2\).
\end{proposition}

\begin{proof}
After the substitution \(f(x)=(1-x)^\alpha u(x)\), the equation becomes Gauss's hypergeometric equation with the above parameters. The general solution is therefore the standard hypergeometric linear combination, and multiplying back by \((1-x)^\alpha\) yields the result.
\end{proof}

\begin{remark}
The local exponents at \(x=0\) are obtained by substituting \(f(x)\sim x^r\), which gives \(2r(2r+1)=0\). Hence the local exponents at \(x=0\) are \(r=0\) and \(r=-\frac12\). At \(x=1\), setting \(t=1-x\) and substituting \(f(x)\sim t^r\) yields the indicial equation
\[
4r^2+4r+1-\lambda=0,
\]
so that
\[
r=\frac{-1\pm\sqrt{\lambda}}{2}.
\]
\end{remark}

\subsection{The eigenvalue equation for \texorpdfstring{$\mathcal D_{\Lambda}$}{DLambda}}

We next consider
\[
\mathcal D_{\Lambda}g=\mu g,
\]
that is,
\[
4x(1-x)^2g''+6(x-1)(3x-1)g'+(12x-8-\mu)g=0.
\]

Set
\[
g(x)=(1-x)^\beta v(x).
\]
Then the equation reduces to hypergeometric form provided
\[
\mu=4(\beta+1)^2.
\]
The function \(v\) satisfies
\[
x(1-x)v''+\left(\frac32-\frac{4\beta+9}{2}x\right)v'
-\left(\beta+\frac32\right)(\beta+2)v=0.
\]
Thus the corresponding parameters are
\[
a=\beta+\frac32,
\qquad
b=\beta+2,
\qquad
c=\frac32.
\]

\begin{proposition}\label{prop:eigen-DLambda}
The general formal solution of the eigenvalue equation
\[
\mathcal D_{\Lambda}g=\mu g
\]
is
\[
g(x)=(1-x)^\beta\left[
C_1\,{}_2F_1\!\left(\beta+\frac32,\beta+2;\frac32;x\right)
+
C_2\,x^{-1/2}\,{}_2F_1\!\left(\beta+1,\beta+\frac32;\frac12;x\right)
\right],
\]
where
\[
\mu=4(\beta+1)^2.
\]
\end{proposition}

\begin{proof}
The proof is identical in spirit to that of \cref{prop:eigen-DXi}. After the substitution \(g(x)=(1-x)^\beta v(x)\), one obtains a Gauss hypergeometric equation with the above parameters, and the general solution follows.
\end{proof}

\begin{remark}
As in the case of \(\mathcal D_{\Xi}\), the local exponents at \(x=0\) are \(r=0\) and \(r=-\frac12\). At \(x=1\), the indicial equation becomes
\[
4(r+1)^2-\mu=0,
\]
hence
\[
r=-1\pm\frac{\sqrt{\mu}}{2}.
\]
\end{remark}

\section{Hyperbolicity}

\begin{lemma}\label{lem:An-hyperbolicity}
For every \(n \in \mathbb{N}_0\), the differential operator
\[
A_n:=2(x-1)D+n
\]
preserves hyperbolicity. In particular, \(A_1\) and \(A_2\) preserve hyperbolicity.
\end{lemma}

\begin{proof}
Let \(f\in \mathbb{R}[x]\) be hyperbolic, and define \(g(y):=f(y+1)\). Since translation preserves real-rootedness, \(g\) is again hyperbolic.

Setting \(x=y+1\), we obtain
\[
A_n[f](y+1)=2y\,g'(y)+n\,g(y).
\]
Thus it suffices to show that the operator
\[
T_n:=2yD+n
\]
preserves hyperbolicity on \(\mathbb{R}[y]\).

For each \(k\ge 0\), we have \(T_n[y^k]=(2k+n)y^k\), so \(T_n\) is diagonal in the monomial basis with diagonal sequence \(\gamma_k=2k+n\) for \(k\ge 0\). Its exponential generating function is
\[
\sum_{k=0}^\infty \gamma_k \frac{z^k}{k!}
=
\sum_{k=0}^\infty (2k+n)\frac{z^k}{k!}
=
(2z+n)e^z.
\]
Since \(n \in \mathbb{N}_0\), the polynomial \(2z+n\) has the unique zero \(-n/2\le 0\). Hence \((2z+n)e^z\) belongs to the Laguerre--Pólya class with only real nonpositive zeros. By the Pólya--Schur theorem (see, for example, \cite{BorceaBranden2009,RahmanSchmeisser2002}), the sequence \(\{\gamma_k\}_{k\ge0}\) is a multiplier sequence. Therefore \(T_n\) preserves hyperbolicity.

Finally, since \(A_n[f](x)=T_n[g](x-1)\), and translation preserves hyperbolicity, it follows that \(A_n[f]\) is hyperbolic. Hence \(A_n\) preserves hyperbolicity.
\end{proof}

\begin{lemma}\label{lem:Bn-hyperbolicity}
For \(n\in\mathbb{N}_0\), the differential operator
\[
B_n:=2x(x-1)D+((n+1)x-1)
\]
preserves hyperbolicity.
\end{lemma}

\begin{proof}
We apply the Borcea--Br\"and\'en criterion for linear operators on \(\mathbb R[x]\). The transcendental symbol of \(B_n\) is
\[
G_{B_n}(x,y)
:=
\sum_{k=0}^\infty \frac{(-1)^k}{k!}\,Q_k(x)y^k
=
((n+1)x-1)-2x(x-1)y,
\]
where \(Q_0(x)=((n+1)x-1)\), \(Q_1(x)=2x(x-1)\), and \(Q_k(x)=0\) for \(k\ge2\).

It therefore suffices to show that the bivariate polynomial
\[
((n+1)x-1)-2x(x-1)y
\]
is stable. Suppose \(x,y\in\mathbb H\), where \(\mathbb H:=\{z\in\mathbb C:\operatorname{Im} z>0\}\), and assume
\[
((n+1)x-1)-2x(x-1)y=0.
\]
Then
\[
y=\frac{(n+1)x-1}{2x(x-1)}
=
\frac{1}{2x}+\frac{n}{2(x-1)}.
\]
Since \(x\in\mathbb H\), we have \(1/x\in-\mathbb H\), and likewise \(1/(x-1)\in-\mathbb H\), because \(x-1\in\mathbb H\). Hence \(\frac{1}{2x}\in-\mathbb H\), while \(\frac{n}{2(x-1)}\) has nonpositive imaginary part for \(n\in\mathbb N_0\). Therefore \(y\) has negative imaginary part, so \(y\notin\mathbb H\), a contradiction.

Thus \(G_{B_n}(x,y)\) is stable. By the Borcea--Br\"and\'en theorem, \(B_n\) preserves stability on \(\mathbb{R}[x]\), and since in one variable stability is equivalent to hyperbolicity, \(B_n\) preserves hyperbolicity.
\end{proof}

\begin{proposition}\label{prop:D-hyperbolicity}
The differential operators \(\mathcal D_{\Xi}\) and \(\mathcal D_{\Lambda}\) preserve hyperbolicity.
\end{proposition}

\begin{proof}
By \cref{prop:factorization}, we have
\[
\mathcal D_{\Xi}=A_1\circ B_1,
\qquad
\mathcal D_{\Lambda}=A_2\circ B_2.
\]
By \cref{lem:An-hyperbolicity}, the operators \(A_1\) and \(A_2\) preserve hyperbolicity, and by \cref{lem:Bn-hyperbolicity}, the operators \(B_1\) and \(B_2\) preserve hyperbolicity as well.

Since the composition of hyperbolicity-preserving operators is again hyperbolicity-preserving, it follows that both
\[
\mathcal D_{\Xi}=A_1\circ B_1
\qquad\text{and}\qquad
\mathcal D_{\Lambda}=A_2\circ B_2
\]
preserve hyperbolicity.
\end{proof}

\section{Preservation of zeros in \texorpdfstring{$(0,b)$}{(0,b)}}

In this section we show that the operators \(\mathcal D_{\Xi}\) and \(\mathcal D_{\Lambda}\) preserve the location of zeros in the interval \((0,b)\), provided that \(b\ge 1\). The argument is based on the factorizations established in \cref{prop:factorization}. We first study the first-order operators
\[
A_n:=2(x-1)D+n,
\qquad
B_n:=2x(x-1)D+((n+1)x-1),
\]
where \(D=\frac{d}{dx}\), and then deduce the result for \(\mathcal D_{\Xi}\) and \(\mathcal D_{\Lambda}\) from the identities
\[
\mathcal D_{\Xi}=A_1\circ B_1,
\qquad
\mathcal D_{\Lambda}=A_2\circ B_2.
\]

\begin{lemma}\label{lem:An-interval-0b}
Let \(n \in \mathbb{N}_0\) and \(b\in\mathbb R\) with \(b\ge 1\). If \(f\in\mathbb R[x]\) is a polynomial all of whose zeros lie in \((0,b)\), then all zeros of \(A_n[f]\) also lie in \((0,b)\).
\end{lemma}

\begin{proof}
Write \(f(x)=c\prod_{j=1}^m (x-\lambda_j)\), where \(\lambda_j\in(0,b)\). Let \(\xi\) be a zero of \(A_n[f]\). If \(f(\xi)=0\), then automatically \(\xi\in(0,b)\), so there is nothing to prove. We may therefore assume that \(f(\xi)\neq 0\). Since
\[
A_n[f]=2(x-1)D[f]+nf,
\]
the identity \(A_n[f](\xi)=0\) yields
\[
2(\xi-1)\frac{D[f](\xi)}{f(\xi)}+n=0.
\]
Using the logarithmic derivative,
\[
\frac{D[f](\xi)}{f(\xi)}
=
\sum_{j=1}^m \frac{1}{\xi-\lambda_j},
\]
we obtain
\[
2(\xi-1)\sum_{j=1}^m \frac{1}{\xi-\lambda_j}+n=0.
\]

Suppose first that \(\xi<0\). Then \(\xi-1<0\) and \(\xi-\lambda_j<0\) for all \(j\), so each summand \(\displaystyle \frac{\xi-1}{\xi-\lambda_j}\) is positive. Hence
\[
2(\xi-1)\sum_{j=1}^m \frac{1}{\xi-\lambda_j}>0,
\]
and therefore the left-hand side is strictly positive, a contradiction.

Next suppose that \(\xi>1\). Since \(b\ge 1\), we have \(\lambda_j<b\le \xi\), so \(\xi-\lambda_j>0\) for all \(j\). Thus
\[
2(\xi-1)\sum_{j=1}^m \frac{1}{\xi-\lambda_j}>0,
\]
which again contradicts the defining equation.

Finally, \(A_n=n f(1)\neq 0\), because \(1\notin(0,b)\) and all zeros of \(f\) lie in \((0,b)\). Hence \(\xi\neq 1\).

We conclude that every zero \(\xi\) of \(A_n[f]\) lies in \((0,b)\).
\end{proof}

\begin{lemma}\label{lem:Bn-interval-0b}
Let \(b\in\mathbb R\) with \(b\ge 1\), and let \(n\in\mathbb N_0\). If \(f\in\mathbb R[x]\) is a polynomial all of whose zeros lie in \((0,b)\), then all zeros of \(B_n[f]\) also lie in \((0,b)\).
\end{lemma}

\begin{proof}
Write
\[
f(x)=c\prod_{j=1}^m (x-\lambda_j),
\qquad \lambda_j\in(0,b).
\]
Let \(\xi\) be a zero of \(B_n[f]\). If \(f(\xi)=0\), then \(\xi\in(0,b)\), so it remains to consider the case \(f(\xi)\neq 0\). Since
\[
B_n[f]=2x(x-1)D[f]+((n+1)x-1)f,
\]
the equation \(B_n[f](\xi)=0\) becomes
\[
2\xi(\xi-1)\frac{D[f](\xi)}{f(\xi)}+((n+1)\xi-1)=0.
\]
Using the logarithmic derivative,
\[
\frac{D[f](\xi)}{f(\xi)}
=
\sum_{j=1}^m \frac{1}{\xi-\lambda_j},
\]
we obtain
\[
2\xi(\xi-1)\sum_{j=1}^m \frac{1}{\xi-\lambda_j}+((n+1)\xi-1)=0.
\]

Assume first that \(\xi<0\). Then \(\xi<0\) and \(\xi-1<0\), so \(\xi(\xi-1)>0\), while \(\xi-\lambda_j<0\) for every \(j\). Hence
\[
2\xi(\xi-1)\sum_{j=1}^m \frac{1}{\xi-\lambda_j}<0.
\]
Moreover, \(((n+1)\xi-1)<0\). Therefore the left-hand side is strictly negative, a contradiction.

Now assume that \(\xi>b\). Since \(\lambda_j\in(0,b)\), we have \(\xi-\lambda_j>0\) for all \(j\). Also \(\xi>b\ge1\), so \(\xi(\xi-1)>0\), and
\[
((n+1)\xi-1)\ge \xi-1>0.
\]
Hence
\[
2\xi(\xi-1)\sum_{j=1}^m \frac{1}{\xi-\lambda_j}+((n+1)\xi-1)>0,
\]
again a contradiction.

Finally, \(B_n=-f(0)\neq0\). If \(b=1\), then
\[
B_n=n\,f(1)\neq0.
\]
If \(b>1\), then at \(\xi=b\) we have \(\xi-\lambda_j>0\) for all \(j\), \(b(b-1)>0\), and \(((n+1)b-1)>0\), so the same identity shows that \(B_n[f](b)\neq0\).

Therefore every zero \(\xi\) of \(B_n[f]\) lies in \((0,b)\).
\end{proof}

We can now combine the previous two lemmas with the factorization of \cref{prop:factorization}.

\begin{proposition}\label{prop:D-interval-0b}
Let \(b\in\mathbb R\) with \(b\ge 1\), and let \(f\in\mathbb R[x]\) be a polynomial all of whose zeros lie in \((0,b)\). Then all zeros of \(\mathcal D_{\Xi}[f]\) and \(\mathcal D_{\Lambda}[f]\) also lie in \((0,b)\).
\end{proposition}

\begin{proof}
By \cref{prop:factorization}, we have
\[
\mathcal D_{\Xi}=A_1\circ B_1,
\qquad
\mathcal D_{\Lambda}=A_2\circ B_2.
\]
By \cref{lem:An-interval-0b}, both \(A_1\) and \(A_2\) map polynomials whose zeros lie in \((0,b)\) to polynomials whose zeros again lie in \((0,b)\). By \cref{lem:Bn-interval-0b}, the same is true for \(B_1\) and \(B_2\). Therefore the compositions \(A_1\circ B_1\) and \(A_2\circ B_2\) preserve the class of polynomials whose zeros lie in \((0,b)\). This proves the claim.
\end{proof}

\begin{remark}
The assumption \(b\ge 1\) is sharp. For \(0<b<1\), the interval \((0,b)\) is not preserved in general.

For example, let \(f(x)=x-\frac14\). Then \(f\) has its unique zero in \((0,b)\) whenever \(b>\frac14\). On the other hand,
\[
A_1[f](x)=3x-\frac94,
\]
so \(A_1[f]\) has the zero \(x=\frac34\). Thus \(A_1\) fails to preserve \((0,b)\) for every \(b\in(\frac14,\frac34)\).

Similarly,
\[
B_1[f](x)=4x^2-\frac72x+\frac14,
\]
whose zeros are
\[
\frac{7\pm\sqrt{33}}{16}\approx 0.0785,\;0.7965.
\]
Hence \(B_1\) fails to preserve \((0,b)\) for every
\[
b\in\left(\frac14,\frac{7+\sqrt{33}}{16}\right).
\]

It follows that no general preservation result can hold for \(\mathcal D_{\Xi}\) on intervals \((0,b)\) with \(0<b<1\). Since \(A_1\) already fails to preserve \((0,b)\) in this range, the same conclusion applies a fortiori to the factorization-based argument.
\end{remark}

\section{Interlacing}
In this section we study interlacing properties of the operators
\[
A_n:=2(x-1)D+n,
\qquad
B_n:=2x(x-1)D+((n+1)x-1),
\]
where \(D=\frac{d}{dx}\). Our goal is to deduce corresponding results for the second-order operators \(\mathcal D_{\Xi}\) and \(\mathcal D_{\Lambda}\) from the factorizations
\[
\mathcal D_{\Xi}=A_1\circ B_1,
\qquad
\mathcal D_{\Lambda}=A_2\circ B_2.
\]

We begin with the family \(A_n\). As the examples suggest, the natural interlacing statement holds when all zeros of the input polynomial lie in the interval \((0,1)\).

\begin{lemma}\label{lem:An-interlacing}
Let \(n>0\), and let \(f\in\mathbb R[x]\) be a polynomial of degree \(m\) with simple zeros
\[
0<\lambda_1<\lambda_2<\cdots<\lambda_m<1.
\]
Then \(A_n[f]\) has degree \(m\), all of its zeros are real and simple, and if
\[
\mu_1<\mu_2<\cdots<\mu_m
\]
denotes the zeros of \(A_n[f]\), then
\[
\lambda_1<\mu_1<\lambda_2<\mu_2<\cdots<\lambda_m<\mu_m<1.
\]
\end{lemma}

\begin{proof}
Since \(A_n[f]=2(x-1)D[f]+nf\) and \(\deg D[f]=m-1\), we have \(\deg A_n[f]=m\).

For each zero \(\lambda_j\) of \(f\), we have
\[
A_n[f](\lambda_j)=2(\lambda_j-1)D[f](\lambda_j).
\]
Because \(0<\lambda_j<1\), the factor \(\lambda_j-1\) is negative for every \(j\). Since the zeros of \(f\) are simple, the signs of \(D[f](\lambda_j)\) alternate with \(j\). Hence
\[
A_n[f](\lambda_j)A_n[f](\lambda_{j+1})<0
\qquad (j=1,\dots,m-1).
\]
Therefore \(A_n[f]\) has at least one zero in each interval \((\lambda_j,\lambda_{j+1})\) for \(j=1,\dots,m-1\).

Moreover, \(A_n=n f(1)\). Now \(f(1)=c\prod_{j=1}^m(1-\lambda_j)\) has the sign of \(c\), whereas \(A_n[f](\lambda_m)=2(\lambda_m-1)D[f](\lambda_m)\) has the opposite sign, because \(\lambda_m-1<0\) and
\[
D[f](\lambda_m)=c\prod_{j=1}^{m-1}(\lambda_m-\lambda_j)
\]
again has the sign of \(c\). Thus
\[
A_n<0,
\]
so \(A_n[f]\) has at least one zero in \((\lambda_m,1)\).

By \cref{lem:An-interval-0b} with \(b=1\), all zeros of \(A_n[f]\) lie in \((0,1)\). We have therefore found at least one zero of \(A_n[f]\) in each of the \(m\) disjoint intervals
\[
(\lambda_1,\lambda_2),\;(\lambda_2,\lambda_3),\;\dots,\;(\lambda_{m-1},\lambda_m),\;(\lambda_m,1).
\]
Since \(\deg A_n[f]=m\), these are all zeros of \(A_n[f]\). Hence each interval contains exactly one zero, and all zeros are simple. This yields
\[
\lambda_1<\mu_1<\lambda_2<\mu_2<\cdots<\lambda_m<\mu_m<1.
\]
\end{proof}

\begin{lemma}\label{lem:Bn-interlacing}
Let \(n\ge 1\), and let \(f\in\mathbb R[x]\) be a polynomial of degree \(m\) with simple zeros
\[
0<\lambda_1<\lambda_2<\cdots<\lambda_m<1.
\]
Then \(B_n[f]\) has degree \(m+1\), all of its zeros are real and simple, and if
\[
\nu_1<\nu_2<\cdots<\nu_{m+1}
\]
denotes the zeros of \(B_n[f]\), then
\[
0<\nu_1<\lambda_1<\nu_2<\lambda_2<\cdots<\nu_m<\lambda_m<\nu_{m+1}<1.
\]
In particular, the zeros of \(B_n[f]\) strictly interlace those of \(f\).
\end{lemma}

\begin{proof}
Write \(f(x)=cx^m+\cdots\) with \(c\neq 0\). Since \(D[f](x)=mcx^{m-1}+\cdots\), we have
\[
2x(x-1)D[f](x)=2mc\,x^{m+1}+\cdots,
\qquad
((n+1)x-1)f(x)=(n+1)c\,x^{m+1}+\cdots.
\]
Hence the coefficient of \(x^{m+1}\) in \(B_n[f]\) is \((2m+n+1)c\), so \(\deg B_n[f]=m+1\).

For each zero \(\lambda_j\) of \(f\), we have
\[
B_n[f](\lambda_j)=2\lambda_j(\lambda_j-1)D[f](\lambda_j).
\]
Since \(0<\lambda_j<1\), the factor \(\lambda_j(\lambda_j-1)\) is negative for every \(j\). Because the zeros of \(f\) are simple, the signs of \(D[f](\lambda_j)\) alternate with \(j\). It follows that the signs of \(B_n[f](\lambda_j)\) also alternate, and therefore
\[
B_n[f](\lambda_j)\,B_n[f](\lambda_{j+1})<0
\qquad (j=1,\dots,m-1).
\]
Thus \(B_n[f]\) has at least one zero in each interval \((\lambda_j,\lambda_{j+1})\) for \(j=1,\dots,m-1\).

We next analyze the boundary intervals. Since
\[
B_n=-f(0),
\qquad
B_n=n\,f(1),
\]
it suffices to compare the signs of these quantities with those of \(B_n[f](\lambda_1)\) and \(B_n[f](\lambda_m)\), respectively.

Now \(f(0)=c\prod_{j=1}^m(-\lambda_j)\) has sign \((-1)^m\operatorname{sgn}(c)\), and therefore \(B_n=-f(0)\) has sign \((-1)^{m+1}\operatorname{sgn}(c)\). On the other hand,
\[
D[f](\lambda_1)=c\prod_{j=2}^m(\lambda_1-\lambda_j)
\]
has sign \((-1)^{m-1}\operatorname{sgn}(c)\). Since \(\lambda_1(\lambda_1-1)<0\), it follows that \(B_n[f](\lambda_1)\) has sign \((-1)^m\operatorname{sgn}(c)\). Hence
\[
B_n\,B_n[f](\lambda_1)<0,
\]
so \(B_n[f]\) has at least one zero in \((0,\lambda_1)\).

Similarly, \(f(1)=c\prod_{j=1}^m(1-\lambda_j)\) has the sign of \(c\), hence \(B_n=n\,f(1)\) also has the sign of \(c\), since \(n\ge1\). Moreover,
\[
D[f](\lambda_m)=c\prod_{j=1}^{m-1}(\lambda_m-\lambda_j)
\]
also has the sign of \(c\), whereas \(\lambda_m(\lambda_m-1)<0\). Therefore \(B_n[f](\lambda_m)\) has the opposite sign, and thus
\[
B_n[f](\lambda_m)\,B_n<0.
\]
So \(B_n[f]\) has at least one zero in \((\lambda_m,1)\).

By \cref{lem:Bn-interval-0b} with \(b=1\), all zeros of \(B_n[f]\) lie in \((0,1)\). We have therefore found at least one zero in each of the \(m+1\) disjoint intervals
\[
(0,\lambda_1),\;(\lambda_1,\lambda_2),\;\dots,\;(\lambda_{m-1},\lambda_m),\;(\lambda_m,1).
\]
Since \(\deg B_n[f]=m+1\), these are all zeros of \(B_n[f]\). Hence each interval contains exactly one zero, and all zeros are simple. This proves that
\[
0<\nu_1<\lambda_1<\nu_2<\lambda_2<\cdots<\nu_m<\lambda_m<\nu_{m+1}<1.
\]
\end{proof}

\begin{remark}
The assumption \(n\ge 1\) in \cref{lem:Bn-interlacing} is necessary. Indeed, for \(n=0\) one has
\[
B_0[f](x)=2x(x-1)f'(x)+(x-1)f(x)=(x-1)\bigl(2x f'(x)+f(x)\bigr),
\]
so \(B_0=0\) for every polynomial \(f\). Thus \(B_0[f]\) always has a zero at \(x=1\), and the strict interlacing conclusion in \((0,1)\) cannot hold in this case. One may at best expect an interlacing statement in which the rightmost zero of \(B_0[f]\) is equal to \(1\).
\end{remark}

\begin{remark}
One might also ask whether, for a polynomial \(f\) with simple zeros in \((0,1)\), the zeros of \(f\) necessarily interlace those of \(\mathcal D_{\Xi}[f]\) or \(\mathcal D_{\Lambda}[f]\). This is not true in general.

For instance, let
\[
f(x)=\Bigl(x-\frac{13}{20}\Bigr)\Bigl(x-\frac{67}{100}\Bigr).
\]
Then the zeros of \(f\) are
\[
0.65 \qquad \text{and} \qquad 0.67.
\]
On the other hand, the zeros of \(\mathcal D_{\Xi}[f]\) are approximately
\[
0.3105,\qquad 0.7888,\qquad 0.9812,
\]
while the zeros of \(\mathcal D_{\Lambda}[f]\) are approximately
\[
0.2819,\qquad 0.7580,\qquad 0.9526.
\]
In both cases, the two zeros of \(f\) lie between the same two consecutive zeros of the transformed polynomial. Hence \(f\) does not interlace \(\mathcal D_{\Xi}[f]\), and \(f\) does not interlace \(\mathcal D_{\Lambda}[f]\).

Thus, although \(B[f]\) interlaces \(f\), and \(\mathcal D_{\Xi}[f]=A_1(B[f])\) as well as \(\mathcal D_{\Lambda}[f]=A_2(B[f])\) interlace \(B[f]\), one cannot in general expect a direct interlacing relation between \(f\) and \(\mathcal D_{\Xi}[f]\) or \(\mathcal D_{\Lambda}[f]\).
\end{remark}

\begin{remark}
The interlacing patterns for \(A_n\) and \(B_n\) are different. If
\[
0<\lambda_1<\cdots<\lambda_m<1
\]
are the zeros of \(f\), then the zeros of \(A_n[f]\) are of the form
\[
\lambda_1<\mu_1<\lambda_2<\mu_2<\cdots<\lambda_m<\mu_m<1,
\]
whereas for \(n\ge 1\), the zeros of \(B_n[f]\) are of the form
\[
0<\nu_1<\lambda_1<\nu_2<\lambda_2<\cdots<\nu_m<\lambda_m<\nu_{m+1}<1.
\]
This reflects the fact that \(A_n\) preserves degree, while \(B_n\) increases degree by one.
\end{remark}

\begin{definition}
Following Borcea--Br\"and\'en, for real polynomials \(f,g\in\mathbb R[x]\) we write
\(
f\ll g
\)
if the polynomial
\(
g+i f
\)
is stable.
\end{definition}

\begin{remark}
In the univariate case, the relation \(f\ll g\) is equivalent to saying that \(f\) and \(g\) have interlacing zeros and that their Wronskian satisfies
\[
W[f,g]:=f'g-fg'\le 0
\]
on \(\mathbb R\).
\end{remark}

\begin{lemma}\label{lem:proper-position-preservation}
The operators \(\mathcal D_{\Xi}\) and \(\mathcal D_{\Lambda}\) preserve proper position. More precisely, for all \(f,g\in\mathbb R[x]\),
\[
f\ll g \implies \mathcal D_{\Xi}[f]\ll \mathcal D_{\Xi}[g],
\qquad
f\ll g \implies \mathcal D_{\Lambda}[f]\ll \mathcal D_{\Lambda}[g].
\]
\end{lemma}

\begin{proof}
By \cref{prop:D-hyperbolicity}, the operators \(\mathcal D_{\Xi}\) and
\(\mathcal D_{\Lambda}\) preserve hyperbolicity. In one variable, this is
equivalent to preserving real stability, so both operators are real stability
preservers in the sense of Borcea--Br\"and\'en.

Hence, by \cite[Theorem~4.3]{BB2009}, from \(f\ll g\) it follows that either
\(\mathcal D_{\Xi}[f]\ll \mathcal D_{\Xi}[g]\) or
\(\mathcal D_{\Xi}[f]=\mathcal D_{\Xi}[g]\equiv 0\), and likewise either
\(\mathcal D_{\Lambda}[f]\ll \mathcal D_{\Lambda}[g]\) or
\(\mathcal D_{\Lambda}[f]=\mathcal D_{\Lambda}[g]\equiv 0\).

It therefore remains to rule out the zero alternative. Let
\(p(x)=a_nx^n+\cdots\) with \(a_n\neq 0\). A direct computation shows that the
leading term of \(\mathcal D_{\Xi}[p]\) is
\(2(n+1)(2n+3)a_n x^{n+1}\), whereas the leading term of
\(\mathcal D_{\Lambda}[p]\) is \(2(2n+3)(n+2)a_n x^{n+1}\). Since these
coefficients are nonzero for every \(n\ge 0\), both operators raise the degree
of every nonzero polynomial by exactly one. In particular,
\(\mathcal D_{\Xi}[p]\neq 0\) and \(\mathcal D_{\Lambda}[p]\neq 0\) whenever
\(p\neq 0\).

Thus neither \(\mathcal D_{\Xi}\) nor \(\mathcal D_{\Lambda}\) can annihilate a
nonzero polynomial, so the alternative
\(\mathcal D_{\Xi}[f]=\mathcal D_{\Xi}[g]\equiv 0\), and similarly
\(\mathcal D_{\Lambda}[f]=\mathcal D_{\Lambda}[g]\equiv 0\), is impossible.
Therefore \(f\ll g\) implies \(\mathcal D_{\Xi}[f]\ll \mathcal D_{\Xi}[g]\),
and likewise \(f\ll g\) implies
\(\mathcal D_{\Lambda}[f]\ll \mathcal D_{\Lambda}[g]\).
\end{proof}

\part{Polynomial sequences associated with \texorpdfstring{$\mathcal D_{\Xi}$ and $\mathcal D_{\Lambda}$}{DXi and DLambda}}

In this part we introduce two auxiliary polynomial families naturally associated with
\(\mathcal D_{\Xi}\) and \(\mathcal D_{\Lambda}\), together with the polynomial
sequences generated from a general linear initial datum \(cx-d\). The auxiliary
families provide a convenient normalized model, while the general iterated sequences
are those that will be studied from the point of view of interlacing.

We define the normalization constants
\[
a_n:=-\frac{1}{8n(2n+1)},
\qquad
b_n:=-\frac{2^{2n+1}-1}{(2^{2n+3}-1)(2n+1)(2n+2)}
\qquad (n\ge1),
\]
and the auxiliary polynomial families \((\widetilde{\Xi}_n)_{n\ge1}\) and
\((\widetilde{\Lambda}_n)_{n\ge1}\) by
\[
\widetilde{\Xi}_1(x):=\frac14,
\qquad
\widetilde{\Xi}_{n+1}(x):=a_n\,\mathcal D_{\Xi}[\widetilde{\Xi}_n](x)
\qquad (n\ge1),
\]
and
\[
\widetilde{\Lambda}_1(x):=\frac17,
\qquad
\widetilde{\Lambda}_{n+1}(x):=b_n\,\mathcal D_{\Lambda}[\widetilde{\Lambda}_n](x)
\qquad (n\ge1).
\]

More generally, given \(c,d\in\mathbb R\), and two real sequences
\((A_n)_{n\ge1}\) and \((B_n)_{n\ge1}\), we consider the polynomial sequences
\((P_n^\Xi)_{n\ge1}\) and \((P_n^\Lambda)_{n\ge1}\) defined by
\[
P_1^\Xi(x):=cx-d,
\qquad
P_{n+1}^\Xi(x):=A_n\,\mathcal D_{\Xi}[P_n^\Xi](x)
\qquad (n\ge1),
\]
and
\[
P_1^\Lambda(x):=cx-d,
\qquad
P_{n+1}^\Lambda(x):=B_n\,\mathcal D_{\Lambda}[P_n^\Lambda](x)
\qquad (n\ge1).
\]

Since \(P_1^\Xi=P_1^\Lambda=cx-d\), we simply write \(P_1\). The next statement expresses the general iterated sequences in terms of the
auxiliary families.

\section{Closed formulae for the polynomials}

\begin{proposition}\label{prop:closed-form-iterates}
For every \(n\ge2\), one has
\[
P_n^\Xi(x)
=
-4\left(\displaystyle\prod_{k=1}^{n-1}\frac{A_k}{a_k}\right)
\left(
\frac{4n(2n+1)}{3}\,c\,\widetilde{\Xi}_{n+1}(x)
+
\left(d-\frac56c\right)\widetilde{\Xi}_n(x)
\right),
\]
and
\[
P_n^\Lambda(x)
=
-7\left(\displaystyle\prod_{k=1}^{n-1}\frac{B_k}{b_k}\right)
\left(
\frac{(2^{2n+3}-1)(2n+1)(2n+2)}{12(2^{2n+1}-1)}\,c\,\widetilde{\Lambda}_{n+1}(x)
+
\left(d-\frac23c\right)\widetilde{\Lambda}_n(x)
\right).
\]
\end{proposition}

\begin{proof}
A direct computation gives
\[
\mathcal D_{\Xi}\!\left[\frac14\right](x)=\frac{6x-5}{4},
\qquad
\mathcal D_{\Lambda}\!\left[\frac17\right](x)=\frac{4(3x-2)}{7},
\]
hence
\[
\widetilde{\Xi}_2(x)
=
a_1\,\mathcal D_{\Xi}\!\left[\widetilde{\Xi}_1\right](x)
=
-\frac{1}{24}\cdot\frac{6x-5}{4}
=
\frac{5-6x}{96},
\]
and
\[
\widetilde{\Lambda}_2(x)
=
b_1\,\mathcal D_{\Lambda}\!\left[\widetilde{\Lambda}_1\right](x)
=
-\frac{7}{372}\cdot\frac{4(3x-2)}{7}
=
\frac{2-3x}{93}.
\]
Therefore,
\[
-16\,\widetilde{\Xi}_2(x)=x-\frac56,
\qquad
-31\,\widetilde{\Lambda}_2(x)=x-\frac23.
\]
It follows that
\[
cx-d
=
c\left(x-\frac56\right)-\left(d-\frac56c\right),
\qquad
cx-d
=
c\left(x-\frac23\right)-\left(d-\frac23c\right).
\]

Iterating the defining recurrences for \(P_n^\Xi\) and \(P_n^\Lambda\), we obtain
\[
P_n^\Xi
=
\left(\displaystyle\prod_{k=1}^{n-1}A_k\right)\mathcal D_\Xi^{\,n-1}[cx-d],
\qquad
P_n^\Lambda
=
\left(\displaystyle\prod_{k=1}^{n-1}B_k\right)\mathcal D_\Lambda^{\,n-1}[cx-d].
\]
Using the above decompositions together with linearity, this yields
\[
P_n^\Xi
=
\left(\displaystyle\prod_{k=1}^{n-1}A_k\right)
\left(
c\,\mathcal D_\Xi^{\,n-1}\!\left[x-\frac56\right]
-
\left(d-\frac56c\right)\mathcal D_\Xi^{\,n-1}[1]
\right),
\]
and
\[
P_n^\Lambda
=
\left(\displaystyle\prod_{k=1}^{n-1}B_k\right)
\left(
c\,\mathcal D_\Lambda^{\,n-1}\!\left[x-\frac23\right]
-
\left(d-\frac23c\right)\mathcal D_\Lambda^{\,n-1}[1]
\right).
\]

Now, from
\[
\widetilde{\Xi}_{m+1}=a_m\,\mathcal D_\Xi[\widetilde{\Xi}_m]
\qquad (m\ge1),
\]
we get
\[
\mathcal D_\Xi[\widetilde{\Xi}_m]=\frac{1}{a_m}\widetilde{\Xi}_{m+1}.
\]
Since \(\widetilde{\Xi}_2=-\frac1{16}\left(x-\frac56\right)\), repeated application gives
\[
\mathcal D_\Xi^{\,n-1}\!\left[x-\frac56\right]
=
-\frac{16}{\displaystyle\prod_{k=2}^{n}a_k}\,\widetilde{\Xi}_{n+1}.
\]
Likewise, from \(\widetilde{\Xi}_1=\frac14\), we obtain
\[
\mathcal D_\Xi^{\,n-1}[1]
=
\frac{4}{\displaystyle\prod_{k=1}^{n-1}a_k}\,\widetilde{\Xi}_n.
\]

Similarly, from
\[
\widetilde{\Lambda}_{m+1}=b_m\,\mathcal D_\Lambda[\widetilde{\Lambda}_m]
\qquad (m\ge1),
\]
we obtain
\[
\mathcal D_\Lambda[\widetilde{\Lambda}_m]=\frac{1}{b_m}\widetilde{\Lambda}_{m+1}.
\]
Since \(\widetilde{\Lambda}_2=-\frac1{31}\left(x-\frac23\right)\), repeated application yields
\[
\mathcal D_\Lambda^{\,n-1}\!\left[x-\frac23\right]
=
-\frac{31}{\displaystyle\prod_{k=2}^{n}b_k}\,\widetilde{\Lambda}_{n+1},
\]
and, because \(\widetilde{\Lambda}_1=\frac17\),
\[
\mathcal D_\Lambda^{\,n-1}[1]
=
\frac{7}{\displaystyle\prod_{k=1}^{n-1}b_k}\,\widetilde{\Lambda}_n.
\]

Substituting these identities into the previous expressions, we obtain
\[
P_n^\Xi(x)
=
-16c\,
\frac{\displaystyle\prod_{k=1}^{n-1}A_k}{\displaystyle\prod_{k=2}^{n}a_k}\,
\widetilde{\Xi}_{n+1}(x)
-
4\left(d-\frac56c\right)
\frac{\displaystyle\prod_{k=1}^{n-1}A_k}{\displaystyle\prod_{k=1}^{n-1}a_k}\,
\widetilde{\Xi}_n(x),
\]
and
\[
P_n^\Lambda(x)
=
-31c\,
\frac{\displaystyle\prod_{k=1}^{n-1}B_k}{\displaystyle\prod_{k=2}^{n}b_k}\,
\widetilde{\Lambda}_{n+1}(x)
-
7\left(d-\frac23c\right)
\frac{\displaystyle\prod_{k=1}^{n-1}B_k}{\displaystyle\prod_{k=1}^{n-1}b_k}\,
\widetilde{\Lambda}_n(x).
\]

Since
\[
\frac{\displaystyle\prod_{k=1}^{n-1}A_k}{\displaystyle\prod_{k=2}^{n}a_k}
=
\left(\displaystyle\prod_{k=1}^{n-1}\frac{A_k}{a_k}\right)\frac{a_1}{a_n},
\qquad
\frac{\displaystyle\prod_{k=1}^{n-1}B_k}{\displaystyle\prod_{k=2}^{n}b_k}
=
\left(\displaystyle\prod_{k=1}^{n-1}\frac{B_k}{b_k}\right)\frac{b_1}{b_n},
\]
and
\[
16\frac{a_1}{a_n}=\frac{4n(2n+1)}{3},
\qquad
\frac{31}{7}\frac{b_1}{b_n}
=
\frac{(2^{2n+3}-1)(2n+1)(2n+2)}{12(2^{2n+1}-1)},
\]
the stated formulas follow.
\end{proof}

\section{Interlacing of zeros}

\begin{proposition}
Let \(c,d \in \mathbb{R}\) with \(c \neq 0\). Then the following hold:
\begin{enumerate}
    \item \(P_1\) and \(P_2^{\Xi}\) have strictly interlacing zeros if and only if
    \(
    \dfrac{3}{7}<\dfrac{d}{c}<1.
    \)
    \item \(P_1\) and \(P_2^{\Lambda}\) have strictly interlacing zeros if and only if
    \(
    \dfrac{1}{3}<\dfrac{d}{c}<1.
    \)
\end{enumerate}
\end{proposition}

\begin{proof}
Since \(c \neq 0\), the polynomial \(P_1(x)=cx-d\) has the unique real zero
\(
r=\dfrac{d}{c}.
\)
Bearing in mind that both \(D_{\Xi}\) and \(D_{\Lambda}\) raise the degree by exactly one, \(P_2^{\Xi}\) and \(P_2^{\Lambda}\) are quadratic polynomials. By \cref{prop:D-hyperbolicity}, these quadratic polynomials have exactly two roots.

We use the following elementary criterion: if \(q\) is a quadratic polynomial with leading coefficient \(a\), then \(q\) has two distinct real zeros and \(r\) lies strictly between them if and only if
\[
a\,q(r)<0.
\]
Indeed, writing \(q(x)=a(x-\alpha)(x-\beta)\) with \(\alpha<\beta\), we have
\[
r\in(\alpha,\beta)
\quad\Longleftrightarrow\quad
(r-\alpha)(r-\beta)<0
\quad\Longleftrightarrow\quad
\frac{q(r)}{a}<0.
\]

Since \(P_1'(x)=c\) and \(P_1''(x)=0\), we obtain
\[
D_{\Xi}[P_1](x)=20cx^2-(25c+6d)x+(6c+5d),
\]
and
\[
D_{\Lambda}[P_1](x)=30cx^2-(32c+12d)x+(6c+8d).
\]
As multiplication by the nonzero constants \(A_1\) and \(B_1\) does not affect the zeros, it suffices to consider
\[
q_{\Xi}(x):=20cx^2-(25c+6d)x+(6c+5d),
\qquad
q_{\Lambda}(x):=30cx^2-(32c+12d)x+(6c+8d).
\]

Evaluating at \(r=d/c\), we get
\[
\begin{cases}
q_{\Xi}\!\left(\dfrac{d}{c}\right)
=2c\left(7\left(\dfrac{d}{c}\right)^2-10\dfrac{d}{c}+3\right),\\[1.2em]
q_{\Lambda}\!\left(\dfrac{d}{c}\right)
=6c\left(3\left(\dfrac{d}{c}\right)^2-4\dfrac{d}{c}+1\right).
\end{cases}
\]
Hence
\[
\begin{cases}
(20c)\,q_{\Xi}\!\left(\dfrac{d}{c}\right)
=40c^2\left(7\left(\dfrac{d}{c}\right)^2-10\dfrac{d}{c}+3\right),\\[1.2em]
(30c)\,q_{\Lambda}\!\left(\dfrac{d}{c}\right)
=180c^2\left(3\left(\dfrac{d}{c}\right)^2-4\dfrac{d}{c}+1\right).
\end{cases}
\]
Since \(40c^2>0\) and \(180c^2>0\), the above criterion yields
\[
\begin{cases}
P_1 \text{ and } P_2^{\Xi} \text{ have strictly interlacing zeros}
\iff
7\left(\dfrac{d}{c}\right)^2-10\dfrac{d}{c}+3<0,\\[1.2em]
P_1 \text{ and } P_2^{\Lambda} \text{ have strictly interlacing zeros}
\iff
3\left(\dfrac{d}{c}\right)^2-4\dfrac{d}{c}+1<0.
\end{cases}
\]
Factoring both quadratic expressions, we obtain
\[
\begin{cases}
7t^2-10t+3=(7t-3)(t-1),\\
3t^2-4t+1=(3t-1)(t-1),
\end{cases}
\qquad\text{where } t=\dfrac{d}{c}.
\]
Therefore,
\[
\begin{cases}
P_1 \text{ and } P_2^{\Xi} \text{ have strictly interlacing zeros}
\iff
\dfrac{3}{7}<\dfrac{d}{c}<1,\\[1.2em]
P_1 \text{ and } P_2^{\Lambda} \text{ have strictly interlacing zeros}
\iff
\dfrac{1}{3}<\dfrac{d}{c}<1.
\end{cases}
\]
This proves the claim.
\end{proof}

\begin{proposition}
Let \(c,d\in\mathbb R\) with \(c\neq 0\), and assume that \(A_n\neq 0\) and \(B_n\neq 0\) for all \(n\ge1\). Then the following hold.
\begin{enumerate}
\item If \(\frac37<\frac dc<1\), then, for every \(n\ge1\), the zeros of \(P_n^\Xi\) and \(P_{n+1}^\Xi\) strictly interlace.

\item If \(\frac13<\frac dc<1\), then, for every \(n\ge1\), the zeros of \(P_n^\Lambda\) and \(P_{n+1}^\Lambda\) strictly interlace.
\end{enumerate}
\end{proposition}

\begin{proof}
We prove the statement for the \(\Xi\)-sequence; the proof for the \(\Lambda\)-sequence is identical.

Assume that \(\frac37<\frac dc<1\). By the previous proposition, the zeros of \(P_1^\Xi\) and \(P_2^\Xi\) strictly interlace.

Since the factors \(A_n\) may have arbitrary sign, we introduce a sign-corrected sequence by defining
\[
\sigma_1:=1,
\qquad
\sigma_{n+1}:=\sigma_n\,\operatorname{sgn}(A_n)
\qquad (n\ge1),
\]
and setting \(Q_n^\Xi:=\sigma_n P_n^\Xi\). Then, using the recurrence \(P_{n+1}^\Xi=A_n\,\mathcal D_\Xi[P_n^\Xi]\), we obtain
\[
Q_{n+1}^\Xi
=
\sigma_{n+1}P_{n+1}^\Xi
=
\sigma_n\operatorname{sgn}(A_n)\,A_n\,\mathcal D_\Xi[P_n^\Xi]
=
|A_n|\,\mathcal D_\Xi[Q_n^\Xi].
\]
Thus the sequence \((Q_n^\Xi)\) satisfies the same recurrence as \((P_n^\Xi)\), but with positive coefficients \(|A_n|>0\).

Moreover, \(Q_1^\Xi=P_1^\Xi\) and \(Q_2^\Xi=\operatorname{sgn}(A_1)P_2^\Xi\), so \(Q_1^\Xi\) and \(Q_2^\Xi\) still have strictly interlacing zeros. Since multiplication by a positive constant preserves proper position, and since \(\mathcal D_\Xi\) preserves proper position by \cref{lem:proper-position-preservation}, an induction on \(n\) yields
\[
Q_n^\Xi \ll Q_{n+1}^\Xi
\qquad\text{for all } n\ge1.
\]
Hence the zeros of \(Q_n^\Xi\) and \(Q_{n+1}^\Xi\) strictly interlace for every \(n\ge1\).

Finally, \(Q_n^\Xi=\pm P_n^\Xi\) for every \(n\), so \(Q_n^\Xi\) and \(P_n^\Xi\) have the same zeros. Therefore the zeros of \(P_n^\Xi\) and \(P_{n+1}^\Xi\) strictly interlace for every \(n\ge1\).

The proof for the \(\Lambda\)-sequence is the same, starting from the condition \(\frac13<\frac dc<1\) and defining
\[
\sigma_1:=1,
\qquad
\sigma_{n+1}:=\sigma_n\,\operatorname{sgn}(B_n),
\qquad
Q_n^\Lambda:=\sigma_n P_n^\Lambda.
\]
Then
\[
Q_{n+1}^\Lambda=|B_n|\,\mathcal D_\Lambda[Q_n^\Lambda],
\]
and the same inductive argument shows that the zeros of \(P_n^\Lambda\) and \(P_{n+1}^\Lambda\) strictly interlace for all \(n\ge1\).
\end{proof}

\section{Asymptotic zero distribution}

\begin{lemma}\label{lem:log-derivative-decomposition}
For \(n\ge 2\), define
\[
\alpha_n^\Xi:=\frac{4n(2n+1)}{3}\,c,
\qquad
\beta^\Xi:=d-\frac56c,
\]
and
\[
\alpha_n^\Lambda:=
\frac{(2^{2n+3}-1)(2n+1)(2n+2)}{12(2^{2n+1}-1)}\,c,
\qquad
\beta^\Lambda:=d-\frac23c.
\]
Moreover, set
\[
R_n^\Xi(z):=\frac{\widetilde{\Xi}_{n+1}(z)}{\widetilde{\Xi}_n(z)},
\qquad
R_n^\Lambda(z):=\frac{\widetilde{\Lambda}_{n+1}(z)}{\widetilde{\Lambda}_n(z)},
\]
whenever these quotients are defined.

Then, for every \(n\ge 2\), one has
\[
P_n^\Xi(z)
=
-4\left(\prod_{k=1}^{n-1}\frac{A_k}{a_k}\right)
\widetilde{\Xi}_n(z)\bigl(\alpha_n^\Xi R_n^\Xi(z)+\beta^\Xi\bigr),
\]
and
\[
P_n^\Lambda(z)
=
-7\left(\prod_{k=1}^{n-1}\frac{B_k}{b_k}\right)
\widetilde{\Lambda}_n(z)\bigl(\alpha_n^\Lambda R_n^\Lambda(z)+\beta^\Lambda\bigr).
\]
Consequently, wherever the expressions are defined, the normalized logarithmic derivatives satisfy
\[
\frac1n\frac{(P_n^\Xi)'(z)}{P_n^\Xi(z)}
=
\frac1n\frac{\widetilde{\Xi}_n'(z)}{\widetilde{\Xi}_n(z)}
+
\frac1n\,
\frac{\alpha_n^\Xi (R_n^\Xi)'(z)}
{\alpha_n^\Xi R_n^\Xi(z)+\beta^\Xi},
\]
and
\[
\frac1n\frac{(P_n^\Lambda)'(z)}{P_n^\Lambda(z)}
=
\frac1n\frac{\widetilde{\Lambda}_n'(z)}{\widetilde{\Lambda}_n(z)}
+
\frac1n\,
\frac{\alpha_n^\Lambda (R_n^\Lambda)'(z)}
{\alpha_n^\Lambda R_n^\Lambda(z)+\beta^\Lambda}.
\]
\end{lemma}

\begin{proof}
By \cref{prop:closed-form-iterates}, for every \(n\ge 2\),
\[
P_n^\Xi(z)
=
-4\left(\prod_{k=1}^{n-1}\frac{A_k}{a_k}\right)
\left(
\frac{4n(2n+1)}{3}\,c\,\widetilde{\Xi}_{n+1}(z)
+
\left(d-\frac56c\right)\widetilde{\Xi}_n(z)
\right).
\]
With the above notation, this becomes
\[
P_n^\Xi(z)
=
-4\left(\prod_{k=1}^{n-1}\frac{A_k}{a_k}\right)
\bigl(
\alpha_n^\Xi \widetilde{\Xi}_{n+1}(z)+\beta^\Xi \widetilde{\Xi}_n(z)
\bigr).
\]
Factoring out \(\widetilde{\Xi}_n(z)\), we obtain
\[
P_n^\Xi(z)
=
-4\left(\prod_{k=1}^{n-1}\frac{A_k}{a_k}\right)
\widetilde{\Xi}_n(z)\bigl(\alpha_n^\Xi R_n^\Xi(z)+\beta^\Xi\bigr),
\]
whenever \(R_n^\Xi(z)\) is defined. Since the prefactor \(\displaystyle -4\prod_{k=1}^{n-1}\frac{A_k}{a_k}\) is independent of \(z\), taking logarithmic derivatives yields
\[
\frac{(P_n^\Xi)'(z)}{P_n^\Xi(z)}
=
\frac{\widetilde{\Xi}_n'(z)}{\widetilde{\Xi}_n(z)}
+
\frac{\bigl(\alpha_n^\Xi R_n^\Xi(z)+\beta^\Xi\bigr)'}
{\alpha_n^\Xi R_n^\Xi(z)+\beta^\Xi}.
\]
As \(\alpha_n^\Xi\) and \(\beta^\Xi\) depend only on \(n,c,d\), not on \(z\), we have
\[
\bigl(\alpha_n^\Xi R_n^\Xi(z)+\beta^\Xi\bigr)'
=
\alpha_n^\Xi (R_n^\Xi)'(z).
\]
Therefore
\[
\frac{(P_n^\Xi)'(z)}{P_n^\Xi(z)}
=
\frac{\widetilde{\Xi}_n'(z)}{\widetilde{\Xi}_n(z)}
+
\frac{\alpha_n^\Xi (R_n^\Xi)'(z)}
{\alpha_n^\Xi R_n^\Xi(z)+\beta^\Xi},
\]
and dividing by \(n\) gives the first claimed identity.

The proof for the \(\Lambda\)-sequence is identical. Indeed, by \cref{prop:closed-form-iterates},
\[
P_n^\Lambda(z)
=
-7\left(\prod_{k=1}^{n-1}\frac{B_k}{b_k}\right)
\left(
\alpha_n^\Lambda \widetilde{\Lambda}_{n+1}(z)
+
\beta^\Lambda \widetilde{\Lambda}_n(z)
\right),
\]
hence
\[
P_n^\Lambda(z)
=
-7\left(\prod_{k=1}^{n-1}\frac{B_k}{b_k}\right)
\widetilde{\Lambda}_n(z)\bigl(\alpha_n^\Lambda R_n^\Lambda(z)+\beta^\Lambda\bigr).
\]
Taking logarithmic derivatives and using again that \(\alpha_n^\Lambda\) and \(\beta^\Lambda\) are constant with respect to \(z\), we obtain
\[
\frac{(P_n^\Lambda)'(z)}{P_n^\Lambda(z)}
=
\frac{\widetilde{\Lambda}_n'(z)}{\widetilde{\Lambda}_n(z)}
+
\frac{\alpha_n^\Lambda (R_n^\Lambda)'(z)}
{\alpha_n^\Lambda R_n^\Lambda(z)+\beta^\Lambda}.
\]
Dividing by \(n\) completes the proof.
\end{proof}

\begin{lemma}\label{lem:ratio-asymptotics}
Let
\[
\Omega:=\mathbb C\setminus(-\infty,1],
\]
fix the branch of \(\sqrt{\cdot}\) on \(\Omega\) satisfying
\[
\Re(\sqrt z)>0 \qquad (z\in\Omega),
\]
and define
\[
u(z):=\frac{\sqrt z-1}{\sqrt z+1},
\qquad z\in\Omega.
\]
For \(n\ge1\), set
\[
R_n^\Xi(z):=\frac{\widetilde{\Xi}_{n+1}(z)}{\widetilde{\Xi}_n(z)},
\qquad
R_n^\Lambda(z):=\frac{\widetilde{\Lambda}_{n+1}(z)}{\widetilde{\Lambda}_n(z)}.
\]
Then, for every \(z\in\Omega\) and every \(n\ge1\),
\[
R_n^\Xi(z)
=
-\frac{(1+\sqrt z)^2}{16(2n)(2n+1)}
\frac{B_{2n+1}(u(z))}{B_{2n-1}(u(z))},
\]
and
\[
R_n^\Lambda(z)
=
-\frac{(2^{2n+1}-1)(1+\sqrt z)^2}{(2^{2n+3}-1)(2n+1)(2n+2)}
\frac{A_{2n+2}(u(z))}{A_{2n}(u(z))}.
\]
Moreover,
\[
R_n^\Xi(z)\longrightarrow -\frac{1}{(\log u(z))^2},
\qquad
R_n^\Lambda(z)\longrightarrow -\frac{1}{(\log u(z))^2},
\qquad n\to\infty,
\]
locally uniformly on \(\Omega\), where \(\log\) denotes the principal branch.
\end{lemma}

\begin{proof}
By the explicit formulae for \(\Xi_n\) and \(\Lambda_n\), one has
\[
\Xi_n(x)
=
\frac{(-1)^{n+1}}{2^{4n-1}(2n-1)!}\,
\frac{(1+x)^{2n-1}}{x}\,
B_{2n-1}\!\left(-\frac{1-x}{1+x}\right),
\]
and
\[
\Lambda_n(x)
=
\frac{(-1)^{n+1}}{(2^{2n+1}-1)(2n)!}\,
\frac{(1+x)^{2n-1}}{x}\,
A_{2n}\!\left(-\frac{1-x}{1+x}\right).
\]
Since \(\widetilde{\Xi}_n(z)=\Xi_n(\sqrt z)\) and \(\widetilde{\Lambda}_n(z)=\Lambda_n(\sqrt z)\), this gives
\[
\widetilde{\Xi}_n(z)
=
\frac{(-1)^{n+1}}{2^{4n-1}(2n-1)!}\,
\frac{(1+\sqrt z)^{2n-1}}{\sqrt z}\,
B_{2n-1}(u(z)),
\]
and
\[
\widetilde{\Lambda}_n(z)
=
\frac{(-1)^{n+1}}{(2^{2n+1}-1)(2n)!}\,
\frac{(1+\sqrt z)^{2n-1}}{\sqrt z}\,
A_{2n}(u(z)).
\]

We first compute the quotient \(R_n^\Xi\). Dividing the expression for \(\widetilde{\Xi}_{n+1}(z)\) by that for \(\widetilde{\Xi}_n(z)\), we obtain
\[
R_n^\Xi(z)
=
\frac{(-1)^{n+2}}{2^{4n+3}(2n+1)!}\,
\frac{(1+\sqrt z)^{2n+1}}{\sqrt z}\,
B_{2n+1}(u(z))
\cdot
\frac{2^{4n-1}(2n-1)!}{(-1)^{n+1}}\,
\frac{\sqrt z}{(1+\sqrt z)^{2n-1}}\,
\frac{1}{B_{2n-1}(u(z))}.
\]
After simplification, \(\displaystyle \frac{(-1)^{n+2}}{(-1)^{n+1}}=-1\), \(\displaystyle \frac{2^{4n-1}}{2^{4n+3}}=\frac1{16}\), \(\displaystyle \frac{(2n-1)!}{(2n+1)!}=\frac{1}{(2n)(2n+1)}\), and \(\displaystyle \frac{(1+\sqrt z)^{2n+1}}{(1+\sqrt z)^{2n-1}}=(1+\sqrt z)^2\), so that
\[
R_n^\Xi(z)
=
-\frac{(1+\sqrt z)^2}{16(2n)(2n+1)}
\frac{B_{2n+1}(u(z))}{B_{2n-1}(u(z))}.
\]

The computation for \(R_n^\Lambda\) is analogous. Indeed,
\[
R_n^\Lambda(z)
=
\frac{(-1)^{n+2}}{(2^{2n+3}-1)(2n+2)!}\,
\frac{(1+\sqrt z)^{2n+1}}{\sqrt z}\,
A_{2n+2}(u(z))
\cdot
\frac{(2^{2n+1}-1)(2n)!}{(-1)^{n+1}}\,
\frac{\sqrt z}{(1+\sqrt z)^{2n-1}}\,
\frac{1}{A_{2n}(u(z))},
\]
hence
\[
R_n^\Lambda(z)
=
-\frac{(2^{2n+1}-1)(1+\sqrt z)^2}{(2^{2n+3}-1)(2n+1)(2n+2)}
\frac{A_{2n+2}(u(z))}{A_{2n}(u(z))}.
\]

We now prove the limits. By the choice of the branch of \(\sqrt{\cdot}\), one has \(|u(z)|<1\) for every \(z\in\Omega\); moreover \(u(z)\notin(-1,0]\). Hence \(u(z)\in D^*:=\{w\in\mathbb C:\ |w|<1\}\setminus(-1,0]\), and we may apply the Eulerian quotient asymptotics at the point \(u(z)\). More precisely,
\[
\frac1m\frac{A_{m+1}(u(z))}{A_m(u(z))}
\longrightarrow
\frac{1-u(z)}{-\log u(z)},
\]
and
\[
\frac1m\frac{B_{m+1}(u(z))}{B_m(u(z))}
\longrightarrow
\frac{2(1-u(z))}{-\log u(z)},
\]
locally uniformly on \(\Omega\). Therefore
\[
\frac{B_{2n+1}(u(z))}{B_{2n-1}(u(z))}
=
\left(\frac1{2n}\frac{B_{2n+1}(u(z))}{B_{2n}(u(z))}\right)
\left(\frac1{2n-1}\frac{B_{2n}(u(z))}{B_{2n-1}(u(z))}\right)
2n(2n-1),
\]
and so
\[
R_n^\Xi(z)
=
-\frac{(1+\sqrt z)^2(2n-1)}{16(2n+1)}
\left(\frac1{2n}\frac{B_{2n+1}(u(z))}{B_{2n}(u(z))}\right)
\left(\frac1{2n-1}\frac{B_{2n}(u(z))}{B_{2n-1}(u(z))}\right).
\]
Passing to the limit yields
\[
R_n^\Xi(z)
\longrightarrow
-\frac{(1+\sqrt z)^2}{16}
\left(\frac{2(1-u(z))}{-\log u(z)}\right)^2
=
-\frac{(1+\sqrt z)^2(1-u(z))^2}{4(\log u(z))^2}.
\]
Since \(\displaystyle 1-u(z)=1-\frac{\sqrt z-1}{\sqrt z+1}=\frac{2}{1+\sqrt z}\), we get \((1+\sqrt z)^2(1-u(z))^2=4\), and hence
\[
R_n^\Xi(z)\longrightarrow -\frac{1}{(\log u(z))^2}
\]
locally uniformly on \(\Omega\).

Likewise,
\[
\frac{A_{2n+2}(u(z))}{A_{2n}(u(z))}
=
\left(\frac1{2n+1}\frac{A_{2n+2}(u(z))}{A_{2n+1}(u(z))}\right)
\left(\frac1{2n}\frac{A_{2n+1}(u(z))}{A_{2n}(u(z))}\right)
(2n+1)(2n),
\]
and therefore
\[
R_n^\Lambda(z)
=
-(1+\sqrt z)^2
\frac{2^{2n+1}-1}{2^{2n+3}-1}
\frac{2n}{2n+2}
\left(\frac1{2n+1}\frac{A_{2n+2}(u(z))}{A_{2n+1}(u(z))}\right)
\left(\frac1{2n}\frac{A_{2n+1}(u(z))}{A_{2n}(u(z))}\right).
\]
Passing to the limit gives
\[
R_n^\Lambda(z)
\longrightarrow
-(1+\sqrt z)^2\cdot\frac14
\left(\frac{1-u(z)}{-\log u(z)}\right)^2
=
-\frac{(1+\sqrt z)^2(1-u(z))^2}{4(\log u(z))^2}
=
-\frac{1}{(\log u(z))^2},
\]
again locally uniformly on \(\Omega\). This completes the proof.
\end{proof}

\begin{lemma}\label{lem:correction-term-vanishes}
Let
\[
\Omega:=\mathbb C\setminus(-\infty,1],
\]
fix the branch of \(\sqrt{\cdot}\) on \(\Omega\) satisfying
\[
\Re(\sqrt z)>0 \qquad (z\in\Omega),
\]
and define
\[
u(z):=\frac{\sqrt z-1}{\sqrt z+1},
\qquad
R(z):=-\frac{1}{(\log u(z))^2},
\qquad z\in\Omega,
\]
where \(\log\) denotes the principal branch.

Assume that
\[
R_n^\Xi(z):=\frac{\widetilde{\Xi}_{n+1}(z)}{\widetilde{\Xi}_n(z)},
\qquad
R_n^\Lambda(z):=\frac{\widetilde{\Lambda}_{n+1}(z)}{\widetilde{\Lambda}_n(z)}
\]
satisfy
\[
R_n^\Xi \longrightarrow R,
\qquad
R_n^\Lambda \longrightarrow R
\]
locally uniformly on \(\Omega\) as \(n\to\infty\). Let
\[
\alpha_n^\Xi:=\frac{4n(2n+1)}{3}\,c,
\qquad
\beta^\Xi:=d-\frac56c,
\]
and
\[
\alpha_n^\Lambda:=
\frac{(2^{2n+3}-1)(2n+1)(2n+2)}{12(2^{2n+1}-1)}\,c,
\qquad
\beta^\Lambda:=d-\frac23c.
\]
If \(c\neq 0\), then
\[
\frac1n\,
\frac{\alpha_n^\Xi (R_n^\Xi)'(z)}
{\alpha_n^\Xi R_n^\Xi(z)+\beta^\Xi}
\longrightarrow 0,
\qquad
\frac1n\,
\frac{\alpha_n^\Lambda (R_n^\Lambda)'(z)}
{\alpha_n^\Lambda R_n^\Lambda(z)+\beta^\Lambda}
\longrightarrow 0,
\]
locally uniformly on \(\Omega\).
\end{lemma}

\begin{proof}
We prove the statement for the \(\Xi\)-sequence; the proof for the \(\Lambda\)-sequence is identical.

Let \(K\subset\Omega\) be compact. Since \(u\) is holomorphic on \(\Omega\), the function \(\displaystyle R(z)=-\frac{1}{(\log u(z))^2}\) is holomorphic on \(\Omega\). Moreover, \(R\) has no zeros on \(\Omega\). Indeed, for \(z\in\Omega\) one has \(|u(z)|<1\), hence \(u(z)\neq 1\); moreover \(u(z)\notin(-1,0]\), so the principal logarithm is well defined and \(\log u(z)\neq 0\). Therefore \(R(z)\neq 0\) for every \(z\in\Omega\).

Since \(R\) is continuous and nowhere vanishing, there exists a constant \(m_K>0\) such that
\[
|R(z)|\ge m_K
\qquad (z\in K).
\]
By the locally uniform convergence \(R_n^\Xi\to R\) on \(\Omega\), there exists \(N_1\in\mathbb N\) such that for all \(n\ge N_1\),
\[
|R_n^\Xi(z)-R(z)|\le \frac{m_K}{2}
\qquad (z\in K).
\]
Hence
\[
|R_n^\Xi(z)|\ge \frac{m_K}{2}
\qquad (z\in K,\ n\ge N_1).
\]

Since each \(R_n^\Xi\) is holomorphic on \(\Omega\) and \(R_n^\Xi\to R\) locally uniformly, Cauchy's integral formula for derivatives implies that \((R_n^\Xi)'\to R'\) locally uniformly on \(\Omega\). In particular, there exist \(M_K>0\) and \(N_2\in\mathbb N\) such that
\[
|(R_n^\Xi)'(z)|\le M_K
\qquad (z\in K,\ n\ge N_2).
\]

Now \(c\neq 0\) implies \(|\alpha_n^\Xi|\to\infty\), since \(\displaystyle \alpha_n^\Xi=\frac{4n(2n+1)}{3}\,c\). Hence there exists \(N_3\in\mathbb N\) such that for all \(n\ge N_3\),
\[
\frac{|\beta^\Xi|}{|\alpha_n^\Xi|}\le \frac{m_K}{4}.
\]
Therefore, for \(z\in K\) and \(n\ge N:=\max\{N_1,N_2,N_3\}\),
\[
\left|\alpha_n^\Xi R_n^\Xi(z)+\beta^\Xi\right|
=
|\alpha_n^\Xi|
\left|R_n^\Xi(z)+\frac{\beta^\Xi}{\alpha_n^\Xi}\right|
\ge
|\alpha_n^\Xi|
\left(
|R_n^\Xi(z)|-\frac{|\beta^\Xi|}{|\alpha_n^\Xi|}
\right)
\ge
|\alpha_n^\Xi|\frac{m_K}{4}.
\]
It follows that for \(z\in K\) and \(n\ge N\),
\[
\left|
\frac1n\,
\frac{\alpha_n^\Xi (R_n^\Xi)'(z)}
{\alpha_n^\Xi R_n^\Xi(z)+\beta^\Xi}
\right|
\le
\frac1n\,
\frac{|\alpha_n^\Xi|\,M_K}{|\alpha_n^\Xi|\,m_K/4}
=
\frac{4M_K}{m_K}\cdot \frac1n.
\]
The right-hand side tends to \(0\) as \(n\to\infty\), uniformly for \(z\in K\). Thus
\[
\frac1n\,
\frac{\alpha_n^\Xi (R_n^\Xi)'(z)}
{\alpha_n^\Xi R_n^\Xi(z)+\beta^\Xi}
\longrightarrow 0
\]
uniformly on \(K\). Since \(K\subset\Omega\) was arbitrary, the convergence is locally uniform on \(\Omega\).

The proof for the \(\Lambda\)-sequence is the same. One uses that \(R_n^\Lambda\to R\) locally uniformly on \(\Omega\), that \(R\) is nowhere vanishing on \(\Omega\), that \((R_n^\Lambda)'\to R'\) locally uniformly on \(\Omega\), and that \(|\alpha_n^\Lambda|\to\infty\) because
\[
\alpha_n^\Lambda=
\frac{(2^{2n+3}-1)(2n+1)(2n+2)}{12(2^{2n+1}-1)}\,c
\]
and \(c\neq 0\). This yields
\[
\frac1n\,
\frac{\alpha_n^\Lambda (R_n^\Lambda)'(z)}
{\alpha_n^\Lambda R_n^\Lambda(z)+\beta^\Lambda}
\longrightarrow 0
\]
locally uniformly on \(\Omega\), as claimed.
\end{proof}

\begin{theorem}\label{thm:log-derivative-general-sequences}
Let
\[
\Omega:=\mathbb C\setminus(-\infty,1],
\]
fix the branch of \(\sqrt{\cdot}\) on \(\Omega\) satisfying
\[
\Re(\sqrt z)>0 \qquad (z\in\Omega),
\]
and define
\[
u(z):=\frac{\sqrt z-1}{\sqrt z+1},
\qquad z\in\Omega,
\]
where \(\log\) denotes the principal branch. Assume that \(c\neq 0\), and for \(n\ge2\) define
\[
s_n^\Xi(z):=\frac1n\frac{(P_n^\Xi)'(z)}{P_n^\Xi(z)},
\qquad
s_n^\Lambda(z):=\frac1n\frac{(P_n^\Lambda)'(z)}{P_n^\Lambda(z)},
\]
whenever these expressions are defined.

Then, as \(n\to\infty\),
\[
s_n^\Xi(z)\longrightarrow
\frac{1}{\sqrt z(1+\sqrt z)}
+
\frac{1}{\sqrt z(1-\sqrt z)}
\left(
u(z)+\frac{1-u(z)}{\log u(z)}
\right),
\]
and
\[
s_n^\Lambda(z)\longrightarrow
\frac{1}{\sqrt z(1+\sqrt z)}
+
\frac{1}{\sqrt z(1-\sqrt z)}
\left(
u(z)+\frac{1-u(z)}{\log u(z)}
\right),
\]
locally uniformly on \(\Omega\).
\end{theorem}

\begin{proof}
We prove the statement for the \(\Xi\)-sequence; the proof for the \(\Lambda\)-sequence is identical.

By \cref{lem:log-derivative-decomposition}, for every \(n\ge2\),
\[
s_n^\Xi(z)
=
\frac1n\frac{\widetilde{\Xi}_n'(z)}{\widetilde{\Xi}_n(z)}
+
\frac1n\,
\frac{\alpha_n^\Xi (R_n^\Xi)'(z)}
{\alpha_n^\Xi R_n^\Xi(z)+\beta^\Xi},
\]
where
\[
R_n^\Xi(z)=\frac{\widetilde{\Xi}_{n+1}(z)}{\widetilde{\Xi}_n(z)},
\qquad
\alpha_n^\Xi=\frac{4n(2n+1)}{3}\,c,
\qquad
\beta^\Xi=d-\frac56c.
\]

By \cref{lem:ratio-asymptotics}, the sequence \(R_n^\Xi\) converges locally uniformly on \(\Omega\) to the holomorphic function
\[
R(z):=-\frac{1}{(\log u(z))^2}.
\]
Hence, by \cref{lem:correction-term-vanishes},
\[
\frac1n\,
\frac{\alpha_n^\Xi (R_n^\Xi)'(z)}
{\alpha_n^\Xi R_n^\Xi(z)+\beta^\Xi}
\longrightarrow 0
\]
locally uniformly on \(\Omega\).

It remains to identify the limit of the principal term \(\displaystyle \frac1n\frac{\widetilde{\Xi}_n'(z)}{\widetilde{\Xi}_n(z)}\). From \cite[Lemma~16.1]{TallaWaffo2026arxiv2602.16761}, one has
\[
\frac{1}{n-1}\frac{\widetilde{\Xi}_n'(z)}{\widetilde{\Xi}_n(z)}
\longrightarrow
\frac{1}{\sqrt z(1+\sqrt z)}
+
\frac{1}{\sqrt z(1-\sqrt z)}
\left(
u(z)+\frac{1-u(z)}{\log u(z)}
\right)
\]
locally uniformly on \(\Omega\). Therefore,
\[
\frac1n\frac{\widetilde{\Xi}_n'(z)}{\widetilde{\Xi}_n(z)}
=
\frac{n-1}{n}
\left(
\frac{1}{n-1}\frac{\widetilde{\Xi}_n'(z)}{\widetilde{\Xi}_n(z)}
\right)
\longrightarrow
\frac{1}{\sqrt z(1+\sqrt z)}
+
\frac{1}{\sqrt z(1-\sqrt z)}
\left(
u(z)+\frac{1-u(z)}{\log u(z)}
\right)
\]
locally uniformly on \(\Omega\), since \((n-1)/n\to1\).

Combining the last two convergences yields
\[
s_n^\Xi(z)\longrightarrow
\frac{1}{\sqrt z(1+\sqrt z)}
+
\frac{1}{\sqrt z(1-\sqrt z)}
\left(
u(z)+\frac{1-u(z)}{\log u(z)}
\right)
\]
locally uniformly on \(\Omega\).

The proof for the \(\Lambda\)-sequence is the same. By \cref{lem:log-derivative-decomposition},
\[
s_n^\Lambda(z)
=
\frac1n\frac{\widetilde{\Lambda}_n'(z)}{\widetilde{\Lambda}_n(z)}
+
\frac1n\,
\frac{\alpha_n^\Lambda (R_n^\Lambda)'(z)}
{\alpha_n^\Lambda R_n^\Lambda(z)+\beta^\Lambda},
\]
and by \cref{lem:ratio-asymptotics,lem:correction-term-vanishes}, the second term tends to \(0\) locally uniformly on \(\Omega\). Moreover, \cite[Lemma~16.1]{TallaWaffo2026arxiv2602.16761} gives
\[
\frac{1}{n-1}\frac{\widetilde{\Lambda}_n'(z)}{\widetilde{\Lambda}_n(z)}
\longrightarrow
\frac{1}{\sqrt z(1+\sqrt z)}
+
\frac{1}{\sqrt z(1-\sqrt z)}
\left(
u(z)+\frac{1-u(z)}{\log u(z)}
\right)
\]
locally uniformly on \(\Omega\), and hence again
\[
\frac1n\frac{\widetilde{\Lambda}_n'(z)}{\widetilde{\Lambda}_n(z)}
\longrightarrow
\frac{1}{\sqrt z(1+\sqrt z)}
+
\frac{1}{\sqrt z(1-\sqrt z)}
\left(
u(z)+\frac{1-u(z)}{\log u(z)}
\right)
\]
locally uniformly on \(\Omega\). This proves the claim.
\end{proof}

\vspace{1cm}

\Cref{thm:log-derivative-general-sequences} shows that the normalized
logarithmic derivatives of the sequences \((P_n^\Xi)_{n\ge2}\) and
\((P_n^\Lambda)_{n\ge2}\) converge locally uniformly on \(\Omega\) to the same
limit as the corresponding normalized logarithmic derivatives of the auxiliary
families \((\widetilde{\Xi}_n)_{n\ge1}\) and
\((\widetilde{\Lambda}_n)_{n\ge1}\). Consequently, all conclusions established
in Part~V of \cite{TallaWaffo2026arxiv2602.16761} for the auxiliary families
carry over verbatim to the sequences \((P_n^\Xi)\) and \((P_n^\Lambda)\).

\begin{corollary}\label{cor:transport-zero-distribution}
Assume that \(c\neq 0\). Then the empirical zero measures associated with
\((P_n^\Xi)_{n\ge2}\) and \((P_n^\Lambda)_{n\ge2}\) converge weakly to the same
probability measure \(\mu\) on \((0,1)\) as in Part~V of
\cite{TallaWaffo2026arxiv2602.16761}.
\end{corollary}

\begin{proof}
By \cref{thm:log-derivative-general-sequences}, the normalized
logarithmic derivatives of \(P_n^\Xi\) and \(P_n^\Lambda\) have the same
locally uniform limit on \(\Omega\) as those of the auxiliary families
\(\widetilde{\Xi}_n\) and \(\widetilde{\Lambda}_n\). Therefore the corresponding
Stieltjes transforms have the same limit, and the weak convergence statement
proved in Part~V of \cite{TallaWaffo2026arxiv2602.16761} carries over unchanged.
\end{proof}

\begin{corollary}\label{cor:transport-density-cdf}
The limiting measure \(\mu\) of \cref{cor:transport-zero-distribution} is
absolutely continuous with respect to Lebesgue measure and has density
\[
\rho(x)=
\begin{cases}
\displaystyle
\frac{2}{\sqrt{x}(1-x)\left(\log^2\!\left(\frac{1-\sqrt{x}}{1+\sqrt{x}}\right)+\pi^2\right)},
& 0<x<1,\\[1em]
0, & x\notin(0,1).
\end{cases}
\]
Its distribution function is given by
\[
F(x)=
\begin{cases}
0, & x\le 0,\\[0.4em]
\displaystyle
\frac{2}{\pi}\arctan\!\left(
\frac{1}{\pi}\log\!\left(\frac{1+\sqrt{x}}{1-\sqrt{x}}\right)
\right),
& 0<x<1,\\[0.8em]
1, & x\ge 1.
\end{cases}
\]
In particular, the empirical distribution functions associated with the real
zeros of \(P_n^\Xi\) and \(P_n^\Lambda\) converge pointwise to \(F\) at every
continuity point of \(F\).
\end{corollary}

\begin{proof}
This is exactly the density and distribution function obtained in Part~V of
\cite{TallaWaffo2026arxiv2602.16761} for the auxiliary families. By
\cref{cor:transport-zero-distribution}, the limiting measure is the same for
\((P_n^\Xi)\) and \((P_n^\Lambda)\), so the same formulae apply here.
\end{proof}

\begin{corollary}\label{cor:transport-quantiles}
Let
\[
0<x_{1,n}^\Xi<\cdots<x_{n,n}^\Xi<1
\qquad\text{and}\qquad
0<x_{1,n}^\Lambda<\cdots<x_{n,n}^\Lambda<1
\]
denote the ordered real zeros of \(P_n^\Xi\) and \(P_n^\Lambda\), respectively.
Then their asymptotic distribution is governed by the same quantile law:
\[
F(x_{k,n}^\Xi)\approx \frac{k}{n},
\qquad
F(x_{k,n}^\Lambda)\approx \frac{k}{n},
\qquad n\to\infty.
\]
Equivalently,
\[
x_{k,n}^\Xi \approx F^{-1}\!\left(\frac{k}{n}\right),
\qquad
x_{k,n}^\Lambda \approx F^{-1}\!\left(\frac{k}{n}\right).
\]
Since \(F\) is explicit, this may be written more concretely as
\[
x_{k,n}^\Xi \approx
\tanh^2\!\left(
\frac{\pi}{2}\tan\!\left(\frac{\pi k}{2n}\right)
\right),
\qquad
x_{k,n}^\Lambda \approx
\tanh^2\!\left(
\frac{\pi}{2}\tan\!\left(\frac{\pi k}{2n}\right)
\right).
\]
\end{corollary}

\begin{proof}
By \cref{cor:transport-density-cdf}, the empirical zero measures of
\(P_n^\Xi\) and \(P_n^\Lambda\) converge weakly to the probability measure with
distribution function \(F\). Since \(P_n^\Xi\) and \(P_n^\Lambda\) have degree
\(n\), each of them has \(n\) real zeros in \((0,1)\), counted with
multiplicity. Therefore the corresponding empirical distribution functions are
\[
\frac{1}{n}\#\{j:\ x_{j,n}^\Xi\le x\}
\qquad\text{and}\qquad
\frac{1}{n}\#\{j:\ x_{j,n}^\Lambda\le x\}.
\]
Hence the usual quantile interpretation yields
\[
F(x_{k,n}^\Xi)\approx \frac{k}{n},
\qquad
F(x_{k,n}^\Lambda)\approx \frac{k}{n},
\]
as \(n\to\infty\). Solving for \(x_{k,n}^\Xi\) and \(x_{k,n}^\Lambda\) gives the
stated approximation by \(F^{-1}(k/n)\). Using the explicit formula for
\(F^{-1}\) then yields
\[
F^{-1}(t)=\tanh^2\!\left(\frac{\pi}{2}\tan\!\left(\frac{\pi t}{2}\right)\right),
\qquad 0<t<1,
\]
which proves the final formulae.
\end{proof}

\section*{Acknowledgments}
The author acknowledges the use of an AI language model for assistance with literature search, presentation of the manuscript, verification of results and clarifying standard probability notions.

\printbibliography

\end{document}

%% file: preamble.tex
\usepackage[margin=0.8in]{geometry}
\usepackage{amsmath, amssymb, amsfonts, mathtools}
\usepackage{amsthm}
\usepackage{mathrsfs}
\usepackage{stmaryrd}
\usepackage{esint}
\usepackage{eqnalign}

\usepackage{graphicx}
\usepackage{enumitem}
\usepackage[most]{tcolorbox}
\usepackage{float}
\usepackage{listings}
\usepackage{array}
\usepackage{booktabs}
\usepackage{xcolor}
\usepackage{emptypage}
\usepackage{tensor}
\usepackage{tabularx}

\usepackage{fancyhdr}
\usepackage{hyperref}
\usepackage[nameinlink,noabbrev]{cleveref}

\usepackage{titlesec}
\usepackage[T1]{fontenc}
\usepackage[utf8]{inputenc}
\usepackage{lmodern}
\usepackage{titling}

\usepackage[backend=biber,style=numeric]{biblatex}
\providecommand{\HeadTitle}{} 
\providecommand{\HeadTitleTwo}{} 
\providecommand{\HeadAuthor}{Luc Ramsès TALLA WAFFO} 

\titleformat{\section}[block]
  {\normalfont\large\bfseries\itshape\centering}
  {§\thesection.}
  {1em}
  {}

\titleformat{\subsection}
  {\normalfont\normalsize\bfseries}
  {\thesubsection}
  {1em}
  {}

\newtheorem{theorem}{Theorem}[section]
\newtheorem{lemma}[theorem]{Lemma}
\newtheorem{proposition}[theorem]{Proposition}
\newtheorem{corollary}[theorem]{Corollary}

\newtheorem{remark}[theorem]{Remark}
\newtheorem{definition}[theorem]{Definition}

\crefname{example}{example}{examples}
\Crefname{example}{Example}{Examples}

\crefname{corollary}{corollary}{corollaries}
\Crefname{corollary}{Corollary}{Corollaries}

\crefname{definition}{definition}{definitions}
\Crefname{definition}{Definition}{Definitions}

\crefname{remark}{remark}{remarks}
\Crefname{remark}{Remark}{Remarks}

\crefname{conjecture}{conjecture}{conjectures}
\Crefname{conjecture}{Conjecture}{Conjectures}

\crefname{lemma}{lemma}{lemmas}
\Crefname{lemma}{Lemma}{Lemmas}

\crefname{proposition}{proposition}{propositions}
\Crefname{proposition}{Proposition}{Propositions}

\crefname{theorem}{theorem}{theorems}
\Crefname{theorem}{Theorem}{Theorems}

\numberwithin{equation}{section}

%% file: font.tex
\usepackage{fontspec}

%% file: titlepage.tex
\thispagestyle{fancy}

\vspace{0.2cm}

\begin{center}
\Large{\HeadTitleTwo}
\end{center}

\hspace{3cm}

\begin{center}
Luc Ramsès TALLA WAFFO \\
Technische Universität Darmstadt\\
Karolinenplatz 5, 64289 Darmstadt, Germany\\
ramses.talla@stud.tu-darmstadt.de\\
\vspace{0.5cm}
February 21, 2026
\end{center}